\documentclass[reqno, 12pt]{amsart}

\usepackage{array}
\usepackage{amsmath}
\usepackage{amsfonts}
\usepackage{amssymb}
\usepackage{enumerate}
\usepackage{amsthm}
\usepackage{amsmath, amscd}
\usepackage{bm}
\usepackage{xy}
\xyoption{all}
\usepackage{color}
\usepackage{marvosym}
\usepackage{amsbsy}
\usepackage{hyperref}
\usepackage{tikz}
\usetikzlibrary{matrix,arrows,backgrounds}%, decorations, pathmorphing, positioning, fit}
	
\newtheorem{theorem}{Theorem}[section]

\newtheorem{lemma}[theorem]{Lemma}
\newtheorem{proposition}[theorem]{Proposition}
\newtheorem{corollary}[theorem]{Corollary}

\theoremstyle{definition}
\newtheorem{definition}[theorem]{Definition}
\newtheorem{remark}[theorem]{Remark}

\newtheorem{example}[theorem]{Example}

\newcommand{\op}[1]{\operatorname{#1}}

\newcommand{\newterm}{\textsf}

\newcommand{\dbcoh}[1]{\operatorname{D}^{\operatorname{b}}(\operatorname{coh }#1)}
\newcommand{\dqcoh}[1]{\operatorname{D}(\operatorname{Qcoh }#1)}

\newcommand{\dabsfact}[1]{\operatorname{D}^{\operatorname{abs}}(\operatorname{fact }#1)}
\newcommand{\dabsFact}[1]{\operatorname{D}^{\operatorname{abs}}(\operatorname{Fact }#1)}
\newcommand{\dcoFact}[1]{\operatorname{D}^{\operatorname{co}}(\operatorname{Fact }#1)}
\newcommand{\dconFact}[1]{\operatorname{D}^{\operatorname{ctr}}(\operatorname{Fact }#1)}

\def\Z{\op{\mathbb{Z}}}

\def\O{\op{\mathcal{O}}}

\def\P{\op{\mathbf{P}}}

\def\tif{\text{if } }

\title[Resolutions in factorization categories]{Resolutions in factorization categories}

\author[Ballard]{Matthew Ballard}
\address{
  \begin{tabular}{l}
   Matthew Ballard  \\ 
      \hspace{.1in} University of South Carolina, Department of Mathematics,  Columbia, SC USA \\
   \hspace{.1in} Email: {\bf ballard@math.sc.edu} \\
  \end{tabular}
}

\author[Deliu]{Dragos Deliu}
\address{
  \begin{tabular}{l}
   Dragos Deliu  \\ 
   \hspace{.1in} Universit\"at Wien, Fakult\"at f\"ur Mathematik,  Wien, \"Osterreich \\
   \hspace{.1in} Email: {\bf dragos.deliu@univie.ac.at} \\
  \end{tabular}
}

\author[Favero]{David Favero}
\address{
  \begin{tabular}{l}
   David Favero \\
   \hspace{.1in} University of Alberta, Department of Mathematics,  Edmonton, AB Canada \\
   \hspace{.1in} Email: {\bf favero@gmail.com} \\
  \end{tabular}
}

\author[Isik]{M. Umut Isik}
\address{
  \begin{tabular}{l}
   M. Umut Isik \\ 
   \hspace{.1in} Universit\"at Wien, Fakult\"at f\"ur Mathematik,  Wien, \"Osterreich \\
   \hspace{.1in} Email: {\bf mehmet.umut.isik@univie.ac.at} \\
  \end{tabular}
}

\author[Katzarkov]{Ludmil Katzarkov}
\address{
  \begin{tabular}{l}
   Ludmil Katzarkov \\ 
   \hspace{.1in} Universit\"at Wien, Fakult\"at f\"ur Mathematik,  Wien, \"Osterreich \\
   \hspace{.1in} Email: {\bf lkatzark@math.uci.edu} \\
  \end{tabular}
}

\numberwithin{equation}{section}
\begin{document}
\renewcommand{\labelenumi}{\emph{\alph{enumi})}}

\begin{abstract}
 Building upon ideas of Eisenbud, Buchweitz, Positselski, and others, we introduce the notion of a factorization category. We then develop some essential tools for working with factorization categories, including constructions of resolutions of factorizations from resolutions of their components and derived functors. Using these resolutions, we lift fully-faithfulness and equivalence statements from derived categories of Abelian categories to derived categories of factorizations. Some immediate geometric consequences include a realization of the derived category of a projective hypersurface as matrix factorizations over a noncommutative algebra and a generalization of a theorem of Baranovsky and Pecharich.
\end{abstract}

\maketitle

\section{Introduction}

Since their introduction by D. Eisenbud \cite{EisMF}, matrix factorizations have spread from commutative algebra into a wide range of fields. In theoretical physics, M. Kontsevich realized that matrix factorizations represent boundary conditions in Landau-Ginzburg models. In topology, matrix factorizations have been used to create knot and link invariants \cite{KR1,KR2}. In algebraic geometry, deep statements tying the geometry of projective hypersurfaces to matrix factorizations of their defining polynomial have been proven by D. Orlov \cite{Orl09}. In addition, through mirror symmetry, matrix factorizations allow access to the structure of Fukaya categories of symplectic manifolds, \cite{Sei08,Efi09,AAEKO,Sheridan}. 
%
%In recent years, the homotopy category of matrix factorizations  and its generalizations \cite{Buc86} have been foundational to new geometric applications of derived categories, namely to recent developments in Homological Projective Duality \cite{BDFIK} and in the relationship between variations of GIT quotients and derived categories \cite{BFK12}. 

The original concept of matrix factorizations can be generalized in various ways, e.g. to the stable module category \cite{Buc86}, the category of singularities \cite{Orl04},  or, in another direction towards more general spaces \cite{Pos1,Pos2, OrlMF}.

Much of the task of this paper is to repackage Positselski's ideas towards a general theory of matrix factorizations for any Abelian category, in particular to derive functors of factorizations as one would functors of Abelian categories. To this end, we introduce the notion of a factorization category for a triple $(\mathcal A, \Phi, w)$ where $\mathcal A$ is an Abelian category, $\Phi: \mathcal A \to \mathcal A$ is an autoequivalence, and $w : \op{Id} \to \Phi$ is a natural transformation.  By appropriately altering $\mathcal A$ and setting $w=0$, one fully recovers the usual construction of the derived category $\op{D}^{\op{b}}(\mathcal A)$. 

As factorization categories can rightly be viewed as a deformation of $\Phi$-twisted, two-periodic chain complexes over $\mathcal A$, one should be able to build resolutions in a straightforward manner from resolutions of the components of a factorization. A key development of this paper is to provide a construction of such resolutions, see Theorems \ref{theorem: strictification injective} and \ref{theorem: strictification projective}.  

Now consider two triples, as above, $(\mathcal A, \Phi, w)$ and $(\mathcal B, \Psi, v)$, and an additive functor, $\theta: \mathcal A \to \mathcal B$, such that
\begin{displaymath}
 \theta \circ \Phi \cong \Psi \circ \theta
\end{displaymath}
and 
\begin{displaymath}
 \theta(w_A) = v_{\theta(A)} : \theta(A) \to \theta(\Phi(A)) \cong \Psi(\theta(A)).
\end{displaymath}
for all objects, $A \in \mathcal A$.  Furthermore, assume that $\theta$ is left-exact, that $\mathcal A$ has small coproducts and enough injectives, and that coproducts of injectives are injective. We can then use these resolutions to prove that if the right derived functor of $\theta$ is fully-faithful then so is a ``right derived functor'' associated to $\theta$ between the derived categories of factorizations.  Moreover, if the right derived functor of $\theta$ is an equivalence induced by Abelian natural transformations,  then so is a ``right derived functor'' associated to $\theta$ between the derived categories of factorizations. We can also use these resolutions to construct a spectral sequence computing the morphism spaces in the derived categories of factorizations whose $E_1$-page consists of Ext-groups between the components of this factorization in the underlying Abelian category. 

From these results, we are able to lay much of the groundwork for working with these categories as one would with derived categories.  Moreover, one can deduce many results about factorization categories from results about the usual derived categories.  Indeed, as a special application to geometry, we provide a derived equivalence between any smooth projective hypersurface and matrix factorizations of a noncommutative algebra. In addition, we generalize the main result of \cite{BP}.

Our work is also foundational to understanding derived categories of gauged Landau-Ginzburg models in algebraic geometry, such as in recent works on variations of GIT quotients and on Homological Projective Duality \cite{Seg2, BFK12, BDFIK}. Indeed, bootstrapping properties of functors from derived categories to factorization categories already appeared in \cite{Seg2, BP}.

\section{Basics} \label{section: basics}

Let $\mathcal A$ be an Abelian category, 
\begin{displaymath}
 \Phi: \mathcal A \to \mathcal A
\end{displaymath}
be an autoequivalence of $\mathcal A$, and 
\begin{displaymath}
 w: \op{Id}_{\mathcal A} \to \Phi
\end{displaymath}
be a natural transformation from the identity functor to $\Phi$. We assume that 
\begin{displaymath}
 w_{\Phi(A)} = \Phi(w_A)
\end{displaymath}
for all $A \in \mathcal A$.

\begin{example} \label{ex: gauged LG model}
 Let $X$ be a smooth algebraic variety and $G$ be an algebraic group acting on $X$.  Consider a $G$-equivariant line bundle $\mathcal L$ on $X$.  Set $\mathcal A$ to be the category of $G$-equivariant coherent sheaves on $X$, $\Phi$ to be tensoring with $\mathcal L$, and $w$ to be a section of $\mathcal L$.
\end{example}

\begin{remark}
The above example is intended to be the category of $B$-branes on a gauged Landau-Ginzburg model. It is the one considered in \cite{BFK11, BFK13, BFK12,  BDFIK}.
\end{remark}

\begin{definition}
 A \newterm{factorization} of the triple, $(A, \Phi, w)$, consists of a pair of objects of $\mathcal A$, $E^{-1}$ and $E^0$, and a pair of morphisms,
 \begin{align*}
  \phi^{-1}_{E} &: \Phi^{-1}(E^{0}) \to E^{-1} \\
  \phi^0_{E} &: E^{-1} \to E^0
 \end{align*}
 such that
 \begin{align*}
  \phi^0_{E} \circ \phi^{-1}_{E} & = \Phi^{-1}(w_{E^0}) : \Phi^{-1}(E^{0}) \to E^0, \\
  \Phi(\phi^{-1}_{E}) \circ \phi_{E}^0 & = w_{E^{-1}} : E^{-1} \to \Phi(E^{-1}), 
 \end{align*}
 We shall often simply denote the factorization, $(E^{-1}, E^0, \phi_{E}^{-1}, \phi_{E}^0)$, by $E$. The objects, $E^0$ and $E^{-1}$, are called the \newterm{components of the factorization}. We also set
 \begin{displaymath}
  E^i := \begin{cases} \Phi^j(E^0) & \text{ if } i=2j \\ \Phi^j(E^{-1}) & \text{ if } i = 2j-1. \end{cases}
 \end{displaymath}
 If $E$ is an object for which $w_E = 0$, then we shall denote the factorization $(0,E,0,0)$ simply by $E$.

 A \newterm{morphism of factorizations}, $g: E \to F$, is a pair of morphisms in $\mathcal A$,
 \begin{align*}
  g^{-1} & : E^{-1} \to F^{-1} \\
  g^0 & : E^0 \to F^0,
 \end{align*}
 making the diagram,
 \begin{center}
 \begin{tikzpicture}[description/.style={fill=white,inner sep=2pt}]
  \matrix (m) [matrix of math nodes, row sep=3em, column sep=3em, text height=1.5ex, text depth=0.25ex]
  {  \Phi^{-1}(E^{0}) & E^{-1} & E^{0} \\
   \Phi^{-1}(F^{0}) & F^{-1} & F^{0} \\ };
  \path[->,font=\scriptsize]
  (m-1-1) edge node[above]{$\phi^{-1}_E$} (m-1-2) 
  (m-1-1) edge node[left]{$\Phi^{-1}(g^{0})$} (m-2-1)
  (m-2-1) edge node[above]{$\phi^{-1}_F$} (m-2-2)
  (m-1-2) edge node[above]{$\phi^{0}_E$} (m-1-3) 
  (m-1-2) edge node[left]{$g^{-1}$} (m-2-2)
  (m-2-2) edge node[above]{$\phi^{-1}_F$} (m-2-3) 
  (m-1-3) edge node[left]{$g^{0}$} (m-2-3)
  ;
 \end{tikzpicture} 
 \end{center}
 commute.
 
 We let $\op{Fact}(w)$ be the category of factorizations. If $\mathcal E$ is a full additive subcategory of $\mathcal A$ preserved by $\Phi$, we let $\op{Fact}(\mathcal E,w)$ be the full subcategory of $\op{Fact}(w)$ consisting of factorizations whose components lie in $\mathcal E$. The most common additive categories we will take are injective objects, where we will use the notation $\op{Fact}(\mathcal Inj \ w)$, and projective objects, where we will use the notation $\op{Fact}(\mathcal Proj \ w)$.
\end{definition}

\begin{lemma} \label{lemma: fact is an abelian category}
 The category, $\op{Fact}(w)$, is Abelian.
\end{lemma}

\begin{proof}
 For a morphism, $g: E \to F$, the componentwise kernel is naturally a factorization, as is the componentwise cokernel. This endows $\op{Fact}(w)$ with the structure of an Abelian category.
\end{proof}

There is a natural notion of translation, or shift, of a factorization.

\begin{definition} 
 Let $\mathcal E$ be a full additive subcategory of $\mathcal A$. Let $[1]$ be the auto-equivalence of $\op{Fact}(\mathcal E,w)$ defined as
 \begin{align*}
  [1]: \op{Fact}(w) & \to \op{Fact}(w) \\
  E & \mapsto E[1] := (E^0,\Phi(E^{-1}), -\phi^0_{E}, -\Phi(\phi^{-1}_{E})) \\
  g & \mapsto g[1] := (g^0,\Phi(g^{-1})).
 \end{align*}
 The functor, $[n]$, is the $n$-fold composition of $[1]$.
\end{definition}

\begin{definition}
 There is also a dg-category associated with factorizations. It is denoted by $\textbf{Fact}(\mathcal E, w)$. The objects are the same as $\op{Fact}(\mathcal E, w)$. Given two factorizations, $E, F \in \textbf{Fact}(w)$, we set 
 \begin{displaymath}
  \op{Hom}_{\textbf{Fact}(w)}^n(E,F) :=\textbf{Hom}_w^n(E,F) := \op{Hom}_{\mathcal A}(E^{-1},(F[n])^{-1}) \oplus \op{Hom}_{\mathcal A}(E^{0},(F[n])^{0}).
 \end{displaymath}
 The differential on $\textbf{Hom}_w^*(E,F)$ takes the pair, $g^{-1}: E^{-1} \to (F[n])^{-1}, g^0: E^0 \to (F[n])^0$, to 
 \begin{align*}
  \phi_{F[n]}^0 \circ g^{-1} - (-1)^n g^0 \circ \phi_E^0 & : E^{-1} \to F[n]^0 = F[n+1]^{-1} \\
  \Phi(\phi_{F[n]}^{-1}) \circ g^0 - (-1)^n \Phi(g^{-1}) \circ \Phi(\phi^{-1}_E) & : E^0 \to \Phi(F[n]^{-1}) = F[n+1]^0.
 \end{align*}
\end{definition}

And a natural cone construction. 

\begin{definition}
 For any morphism, $g : E \to F$, we write, $C(g)$, for the factorization defined as 
\begin{displaymath}
 C(g):= \left( E^{0} \oplus F^{-1}, \Phi(E^{-1}) \oplus F^0, \begin{pmatrix} -\phi_{E}^0 & 0 \\ g^{-1} & \phi_{F}^{-1} \end{pmatrix}, \begin{pmatrix} -\Phi(\phi_{E}^{-1}) & 0 \\ g^0 & \phi_{F}^0 \end{pmatrix} \right).
\end{displaymath}
\end{definition}

\begin{definition}
 A \newterm{homotopy}, $h$, between two morphisms, $g_1,g_2: E \to F$, is a pair of morphisms,
 \begin{align*}
  h^{-1} & : E^{-1} \to \Phi^{-1}(F^0) \\
  h^{0} & : E^{0} \to F^{-1},
 \end{align*}
 such that 
 \begin{align*}
  g_1^{-1}-g_2^{-1} & = h^0 \circ \phi_E^0 + \phi_F^{-1} \circ h^{-1} \\
  g_1^{0}-g_2^{0} & = \phi_F^0 \circ h^0 + \Phi(h^{-1}) \circ \Phi(\phi^{-1}_E).
 \end{align*}
 
 We let $\op{K}(\op{Fact}\mathcal E, w)$ be the homotopy category of $\op{Fact}(\mathcal E, w)$. Note that the homotopy category of the dg-category, $\textbf{Fact}(\mathcal E, w)$, is $\op{K}(\op{Fact}\mathcal E, w)$.
\end{definition}

\begin{proposition} \label{prop: Kfact is a triangulated category}
 The translation, $[1]$, and cones defined above give $\op{K}(\op{Fact}\mathcal E, w)$ the structure of a triangulated category. 
\end{proposition}

\begin{proof}
 This is completely analogous to the standard proof that homotopy categories of chain complexes are triangulated so we refer the reader to \cite[Chapter 4]{GM}.
\end{proof}

\begin{definition} \label{definition: totalizations}
 Let
 \begin{equation} \label{equation: chain complex of factorizations}
  \begin{tikzpicture}[description/.style={fill=white,inner sep=2pt}]
  \matrix (m) [matrix of math nodes, row sep=3em, column sep=3em, text height=1.5ex, text depth=0.25ex]
  { \cdots & E_s & E_{s+1} & \cdots & E_t & E_{t+1} & \cdots \\  };
  \path[->,font=\scriptsize]
  (m-1-1) edge node[above]{$g_{s}$} (m-1-2) 
  (m-1-2) edge node[above]{$g_{s+1}$} (m-1-3)
  (m-1-3) edge node[above]{$g_{s+2}$} (m-1-4)
  (m-1-4) edge node[above]{$g_{t}$} (m-1-5)
  (m-1-5) edge node[above]{$g_{t+1}$} (m-1-6) 
  (m-1-6) edge node[above]{$g_{t+2}$} (m-1-7) 
  ;
  \end{tikzpicture} 
 \end{equation}
 be a complex of factorizations, i.e. a sequence of morphisms in $\op{Fact}(w)$ satisfying
 \begin{displaymath}
  g_{i+1} \circ g_i = 0
 \end{displaymath}
 for all $i \in \Z$. We have two ways to totalize this complex of factorizations into new individual factorization. 
 
 The \newterm{$\bigoplus$-totalization} of Equation~\eqref{equation: chain complex of factorizations} is the factorization $\op{tot}^{\bigoplus}(E_{\bullet}) = T$ whose components are given by the formula
 \begin{displaymath}
  T = \bigoplus_{l \in \Z} E_l[-l]
 \end{displaymath}
 precisely
 \begin{align*}
  T^{-1} & = \bigoplus_{2l} \Phi^{-l}(E_{2l}^{-1}) \oplus \bigoplus_{2l+1} \Phi^{-l-1}(E_{2l+1}^0) \\
  T^0 & = \bigoplus_{2l} \Phi^{-l}(E_{2l}^{0}) \oplus \bigoplus_{2l+1} \Phi^{-l}(E_{2l+1}^{-1}).
 \end{align*}
 The morphisms $\phi^{-1}_T,\phi^0_T$ defining $T$ are determined uniquely by the conditions
 \begin{align*}
  \phi^{-1}_T|_{\Phi^{-l-1}(E^0_{2l})} = \Phi^{-l}(\phi^{-1}_{E_{2l}}) \oplus \Phi^{-l-1}(g^{0}_{2l+1}) & : \Phi^{-l-1}(E^0_{2l}) \to T^{-1} \\
  \phi^{-1}_T|_{\Phi^{-l-1}(E^{-1}_{2l+1})} = -\Phi^{-l-1}(\phi^{0}_{E_{2l+1}}) \oplus \Phi^{-l-1}(g^{-1}_{2l+2}) & : \Phi^{-l-1}(E^{-1}_{2l+1}) \to T^{-1} \\ 
  \phi^{0}_T|_{\Phi^{-l}(E^{-1}_{2l})} = \Phi^{-l}(\phi^{0}_{E_{2l}}) \oplus \Phi^{-l}(g^{-1}_{2l+1}) & : \Phi^{-l}(E^{-1}_{2l}) \to T^{0} \\
  \phi^{0}_T|_{\Phi^{-l-1}(E^{0}_{2l+1})} = -\Phi^{-l}(\phi^{-1}_{E_{2l+1}}) \oplus \Phi^{-l-1}(g^{0}_{2l+2}) & : \Phi^{-l-1}(E^{-1}_{2l+1}) \to T^{0}.
 \end{align*}

 The \newterm{$\prod$-totalization} of Equation~\eqref{equation: chain complex of factorizations} is the factorization $\op{tot}^{\prod}(E_{\bullet}) = T$ whose components are given by the formula
 \begin{displaymath}
  T = \prod_{l \in \Z} E_l[-l]
 \end{displaymath}
 \begin{align*}
  T^{-1} & = \prod_{2l} \Phi^{-l}(E_{2l}^{-1}) \times \prod_{2l+1} \Phi^{-l-1}(E_{2l+1}^0) \\
  T^0 & = \prod_{2l} \Phi^{-l}(E_{2l}^{0}) \times \prod_{2l+1} \Phi^{-l}(E_{2l+1}^{-1}).
 \end{align*}
 The morphisms $\phi^{-1}_T,\phi^0_T$ defining $T$ are determined uniquely by the conditions that 
 \begin{align*}
  \pi_{2l}^{-1} \circ \phi^{-1}_T = (\Phi^{-l}(\phi^{-1}_{E_{2l}}) + \Phi^{-l}(g^{-1}_{2l})) \circ (\Phi^{-1}(\pi_{2l}^0) \oplus \Phi^{-1}(\pi^{0}_{2l-1})) & : \Phi^{-1}(T^0) \to \Phi^{-l}(E^{-1}_{2l}) \\
  \pi_{2l+1}^{-1} \circ \phi^{-1}_T = (-\Phi^{-l-1}(\phi^{0}_{E_{2l+1}}) + \Phi^{-l-1}(g^{0}_{2l+1})) \circ (\Phi^{-1}(\pi_{2l}^0) \oplus \Phi^{-1}(\pi^{0}_{2l-1})) & : \Phi^{-1}(T^0) \to \Phi^{-l-1}(E^{0}_{2l+1}) \\
  \pi_{2l}^{0} \circ \phi^{0}_T = (\Phi^{-l}(\phi^{0}_{E_{2l}}) + \Phi^{-l}(g^{0}_{2l})) \circ (\pi_{2l}^{-1} \oplus \pi^{-1}_{2l-1}) & : T^{-1} \to \Phi^{-l}(E^{0}_{2l}) \\
  \pi_{2l+1}^{0} \circ \phi^{0}_T = (-\Phi^{-l-1}(\phi^{-1}_{E_{2l+1}}) + \Phi^{-l}(g^{-1}_{2l+1})) \circ (\pi_{2l}^{-1} \oplus \pi^{-1}_{2l-1}) & : T^{-1} \to \Phi^{-l}(E^{-1}_{2l+1}) 
 \end{align*}
 where $\pi^l_k$ denotes the projection onto the $k$-th component of $T^{l}$.
 
 If the complex from Equation~\eqref{equation: chain complex of factorizations} is bounded, then the $\bigoplus$-totalization and the $\prod$-totalization coincide. In this case, we call the result simply the \newterm{totalization} and denote it by $\op{tot}(E_{\bullet})$.
 
 Note that the two forms of totalization extend naturally to provide exact functors 
 \begin{displaymath}
  \op{tot}^{\bigoplus},\op{tot}^{\prod}: \op{Ch}(\op{Fact} w) \to \op{Fact}(w).
 \end{displaymath}
\end{definition}

These definitions are due to Positselski, see \cite{Pos1,Pos2}.

\begin{definition}
 Let $\mathcal E$ be a full additive subcategory of $\mathcal A$ preserved by $\Phi$. A factorization is called \newterm{$\mathcal E$-totally acyclic} if it lies in the smallest thick subcategory of $\op{K}(\op{Fact} \mathcal E,w)$ containing the totalizations of all bounded exact complexes from $\op{Fact}(\mathcal E, w)$. We let $\op{Acycl}(\mathcal E, w)$ denote the smallest thick subcategory of $\op{K}(\op{Fact}(\mathcal E, w))$ consisting of $\mathcal E$-totally acyclic factorizations. The \newterm{absolute derived category of $\mathcal E$-factorizations} of $(\mathcal E, \mathcal A, \Phi, w)$ is the Verdier quotient,
 \begin{displaymath}
  \dabsFact{\mathcal E, w} := \op{K}( \op{Fact} \mathcal E, w )/\op{Acycl}(\mathcal E, w).
 \end{displaymath}

 A morphism in $\op{Fact}(\mathcal E, w)$ which becomes an isomorphism in $\dabsFact{\mathcal E, w}$ will be called a \newterm{quasi-isomorphism}, in analogy with the usual derived category. Similarly, two factorizations which are isomorphic in $\dabsFact{\mathcal E, w}$ are called \newterm{quasi-isomorphic}. In the case where $\mathcal E = \mathcal A$, we will simply omit $\mathcal A$ from the notation.
\end{definition}

\begin{definition}
 Let $\mathcal E$ be a full additive subcategory of $\mathcal A$ preserved by $\Phi$. Assume that small coproducts exist in $\mathcal A$. A factorization is called \newterm{$\mathcal E$ co-acyclic} if it lies in the smallest thick subcategory of $\op{K}(\op{Fact} \mathcal E, w)$ containing the totalizations of all bounded exact complexes from $\op{Fact}(\mathcal E, w)$ and closed under taking small coproducts. We let $\op{Co-acycl}(\mathcal E, w)$ denote the thick subcategory of $\op{K}(\op{Fact} \mathcal E, w)$ consisting of $\mathcal E$ co-acyclic factorizations. The \newterm{co-derived category of $\mathcal E$-factorizations} of $(\mathcal E,\mathcal A, \Phi, w)$ is the Verdier quotient,
 \begin{displaymath}
  \dcoFact{\mathcal E, w} := \op{K}( \op{Fact} \mathcal E, w )/\op{Co-acycl}(\mathcal E, w).
 \end{displaymath}

 A morphism in $\op{Fact}(\mathcal E, w)$ which becomes an isomorphism in $\dcoFact{\mathcal E, w}$ will be called a \newterm{co-quasi-isomorphism}. Similarly, two factorizations which are isomorphic in $\dcoFact{\mathcal E, w}$ are called \newterm{co-quasi-isomorphic}. In the case where $\mathcal E = \mathcal A$, we will simply omit $\mathcal A$ from the notation.
\end{definition}

\begin{definition}
 Let $\mathcal E$ be a full additive subcategory of $\mathcal A$ preserved by $\Phi$. Assume that small products exist in $\mathcal A$. A factorization is called \newterm{$\mathcal E$ contra-acyclic} if it lies in the smallest thick subcategory of $\op{K}(\op{Fact} \mathcal E , w)$ containing the totalizations of all bounded exact complexes from $\op{Fact}(\mathcal E, w)$ and closed under taking small products. We let $\op{Ctr-acycl}(\mathcal E,w)$ denote the thick subcategory of $\op{K}(\op{Fact} \mathcal E, w)$ consisting of acyclic factorizations.  The \newterm{contra-derived category of factorizations} of $(\mathcal E, \mathcal A, \Phi, w)$ is the Verdier quotient,
 \begin{displaymath}
  \dconFact{\mathcal E, w} := \op{K}( \op{Fact} \mathcal E, w )/\op{Ctr-acycl}(\mathcal E, w).
 \end{displaymath}

 A morphism in $\op{Fact}(\mathcal E, w)$ which becomes an isomorphism in $\dconFact{\mathcal E, w}$ will be called a \newterm{contra-quasi-isomorphism}. Similarly, two factorizations which are isomorphic in $\dcoFact{\mathcal E, w}$ are called \newterm{contra-quasi-isomorphic}. In the case where $\mathcal E = \mathcal A$, we will simply omit $\mathcal A$ from the notation.
\end{definition}

\begin{example}
 Let $\mathcal A$ be an Abelian category.  Let $\mathcal A^{\op{b}}$ be the category consisting of countably many objects  $a_i \in \mathcal A$ indexed by $\Z$ such that $a_i =0$ for all but finitely many $i$.  Let $\Phi : A^{\op{b}} \to A^{\op{b}}$ be the autoequivalence which shifts the indexing i.e. $\Phi(a)_i = a_{i-1}.$  Let $w=0$.  Then $\op{Fact}(0)$ is equal to $\op{Ch}^{\op{b}}(\mathcal A)$, the category of bounded complexes in $\mathcal A$.

 Furthermore, $\op{Acycl}(0)$ is nothing more than bounded acyclic complexes.  Hence the usual bounded derived category of $\mathcal A$ is nothing more than $\dabsFact{0} \cong \op{D}^{\op{b}}(\mathcal A)$.
\end{example}

\begin{remark}
 Let us attempt to provide some motivation for such definitions. Let us consider the derived category, $\op{D}(\mathcal A)$. It is the localization of $\op{K}(\mathcal A)$ at the class of quasi-isomorphisms. It can also be viewed as the Verdier quotient of $\op{K}(\mathcal A)$ by acyclic complexes. 
 
 How does one make an acyclic complex? One way is to take an exact sequence of complexes,
 \begin{displaymath}
  0 \to E_1 \to E_2 \to E_3 \to 0,
 \end{displaymath}
 and totalize the complex to get an object of $\op{Ch}(\mathcal A)$. This method of construction is fairly robust. Indeed, any finite acyclic complex is easily seen to be the totalization of an exact sequence of chain complexes. These are exactly the analogs of totally-acyclic factorizations. Thus, quotienting by totally-acyclic factorizations should be viewed as the analog of quotienting $\op{K}(\mathcal A)$ by the thick subcategory of finite acyclic complexes.

 To deal with unbounded complexes, we have to take some form of limit of totalizations of bounded exact complexes. Choice of direction of this limit naturally forces one to study infinite products or coproducts of bounded exact complexes. This connection motivates the definitions of co-acyclic and contra-acyclic complexes.
\end{remark}

\begin{lemma} \label{lemma: infinite tots are acyclic}
 Let 
 \begin{center} 
  \begin{tikzpicture}[description/.style={fill=white,inner sep=2pt}]
  \matrix (m) [matrix of math nodes, row sep=3em, column sep=3em, text height=1.5ex, text depth=0.25ex]
  { \cdots & E_s & E_{s+1} & \cdots & E_t & E_{t+1} & \cdots \\ };
  \path[->,font=\scriptsize]
  (m-1-1) edge node[above]{$g_{s}$} (m-1-2) 
  (m-1-2) edge node[above]{$g_{s+1}$} (m-1-3)
  (m-1-3) edge node[above]{$g_{s+2}$} (m-1-4)
  (m-1-4) edge node[above]{$g_{t}$} (m-1-5)
  (m-1-5) edge node[above]{$g_{t+1}$} (m-1-6) 
  (m-1-6) edge node[above]{$g_{t+2}$} (m-1-7) 
  ;
  \end{tikzpicture} 
 \end{center}
 be an exact complex over $\op{Fact}(w)$. If $\mathcal A$ possesses small coproducts and the complex $E_{\bullet}$ is bounded below, then $\op{tot}^{\bigoplus}(E_{\bullet})$ is co-acyclic. If $\mathcal A$ possesses small products and the complex $E_{\bullet}$ is bounded above, then $\op{tot}^{\prod}(E_{\bullet})$ is contra-acyclic.
\end{lemma}

\begin{proof}
 Assume that $\mathcal A$ possesses small coproducts and that $E_{\bullet}$ bounded below. Note that shifting the $E_{\bullet}$ and applying any of the totalizations yields a shift of totalization. So we may assume that $E_s = 0$ for $s < 0$. Let $C_s$ be the kernel of $g_{s}$ so that we have a bounded exact sequence
 \begin{displaymath}
  0 \to E_0 \to E_1 \to \cdots \to E_{s-1} \to E_s \to C_s \to 0. 
 \end{displaymath}
 Denote this complex by $\tau_{\leq s} E_{\bullet}$. Note that there is a natural chain map $h_{s+1}: \tau_{\leq s} E_{\bullet} \to \tau_{\leq s+1} E_{\bullet}$ 
  \begin{center} 
  \begin{tikzpicture}[description/.style={fill=white,inner sep=2pt}]
  \matrix (m) [matrix of math nodes, row sep=3em, column sep=3em, text height=1.5ex, text depth=0.25ex]
  { 0 & E_0 & \cdots & E_{s-1} & E_s & C_{s} & 0 &  \\
    0 & E_0 & \cdots & E_{s-1} & E_s & E_{s+1} & C_{s+1} & 0 \\
  };
  \path[->,font=\scriptsize]
  (m-1-1) edge (m-1-2) 
  (m-1-2) edge (m-1-3)
  (m-1-3) edge (m-1-4)
  (m-1-4) edge (m-1-5)
  (m-1-5) edge (m-1-6) 
  (m-1-6) edge (m-1-7) 
  
  (m-1-2) edge (m-2-2)
  (m-1-4) edge (m-2-4)
  (m-1-5) edge (m-2-5)
  (m-1-6) edge (m-2-6) 
  (m-1-7) edge (m-2-7)
  
  (m-2-1) edge (m-2-2) 
  (m-2-2) edge (m-2-3)
  (m-2-3) edge (m-2-4)
  (m-2-4) edge (m-2-5)
  (m-2-5) edge (m-2-6) 
  (m-2-6) edge (m-2-7)
  (m-2-7) edge (m-2-8)
  ;
  \end{tikzpicture} 
 \end{center}
 One can check that $E_{\bullet}$ is isomorphic to the cokernel of the monomorphism 
 \begin{displaymath}
  \bigoplus_{s \geq 1} \tau_{\leq s} E_{\bullet} \to \bigoplus_{s\geq 1} \tau_{\leq s} E_{\bullet}
 \end{displaymath}
 determined by 
 \begin{displaymath}
  \tau_{\leq s} E_{\bullet} \overset{\op{id}_{\tau_{\leq s} E_{\bullet}} \oplus -h_{s+1}}{\to} \tau_{\leq s} E_{\bullet} \oplus \tau_{\leq s+1} E_{\bullet}.
 \end{displaymath}
 Applying $\op{tot}^{\bigoplus}$ yields an exact sequence of factorizations
 \begin{displaymath}
  0 \to \bigoplus_{s \geq 1} \op{tot}(\tau_{\leq s} E_{\bullet}) \to \bigoplus_{s \geq 1} \op{tot}(\tau_{\leq s} E_{\bullet}) \to \op{tot}^{\bigoplus}(E_{\bullet}) \to 0
 \end{displaymath}
 showing that $\op{tot}^{\bigoplus}(E_{\bullet})$ is co-acyclic. The proof of the other statement is completely analogous and therefore suppressed. 
\end{proof}

\begin{lemma} \label{lemma: Dfact is triang}
 The categories, $\dabsFact{\mathcal E, w}, \dcoFact{\mathcal E, w}, \dconFact{\mathcal E, w}$, with the shift and triangles inherited from $\op{K}(\op{Fact} \mathcal E, w)$, are triangulated categories.
\end{lemma}

\begin{proof}
 Each of these categories is a Verdier quotient of a triangulated category by a thick triangulated subcategory hence triangulated \cite[\S 3]{Ve}.
\end{proof}

Next we demonstrate that, under familiar conditions, many of these categories coincide.

\begin{proposition} \label{proposition: adapted co-replacement}
 Assume that $\mathcal A$ has small coproducts. Let $\mathcal E$ be an additive full subcategory of $\mathcal A$ preserved by $\Phi$ and satisfying the following conditions:
 \begin{itemize}
  \item $\mathcal E$ is closed under coproducts. 
  \item For any object $A \in \mathcal A$, there exists a monomorphism
  \begin{displaymath}
   A \to E
  \end{displaymath}
  with $E$ an object of $\mathcal E$.
 \end{itemize}
 Then, the composition
 \begin{displaymath}
  \op{K}(\op{Fact} \mathcal E,w) \to \op{K}(\op{Fact} w) \to \op{D}^{\op{co}}(\op{Fact} w)
 \end{displaymath}
 induces an equivalence
 \begin{displaymath}
  \op{Q}_{\mathcal E} : \op{D}^{\op{co}}(\op{Fact} \mathcal E,w) \to \op{D}^{\op{co}}(\op{Fact} w).
 \end{displaymath}
\end{proposition}

\begin{proof}
 We first check that any factorization is quasi-isomorphic to a factorization whose components lie in $\mathcal E$. The argument is contained in the proof of \cite[Theorem 3.6]{Pos1}. Let $F$ be a factorization of $w$. By assumption, we may choose objects of $\mathcal E$, $E^{-1}$ and $E^0$, and monomorphisms
 \begin{align*}
  F^{-1} & \overset{f^{-1}}{\to} E^{-1} \\
  F^0 & \overset{f^0}{\to} E^0.
 \end{align*}
 Form the factorization, $G^-(E)$,
 \begin{displaymath}
  \Phi^{-1}(E^0) \oplus E^{-1} \overset{\begin{pmatrix} 0 & 1_{E^{-1}} \\ w_{\Phi^{-1}(E^0)} & 0 \end{pmatrix}}{\to} E^{-1} \oplus E^0 \overset{\begin{pmatrix} 0 & w_{E^{-1}} \\ 1_{E^0} & 0 \end{pmatrix}}{\to} E^0 \oplus \Phi(E^{-1}).
 \end{displaymath}
 The maps
 \begin{align*}
  F^{-1} & \overset{ f^{-1} \oplus f^0 \circ \phi_F^{-1} }{\to} E^{-1} \oplus E^0 \\
  F^0 & \overset{ f^0 \oplus \Phi(f^{-1}) \circ \Phi(\phi_F^0) }{\to} E^0 \oplus \Phi(E^{-1}),
 \end{align*}
 give a monomorphism $F \to G^-(E)$ in $\op{Fact}(w)$. Thus, for any factorization $F$, there exists a factorization with $\mathcal E$-components which received a monomorphism from $F$. We can construct an exact complex of objects of $\op{Fact}(w)$
 \begin{displaymath}
  0 \to F \to E_0 \to E_1 \to \cdots \to E_s \to \cdots
 \end{displaymath}
 where each $E_j$ is a factorization with $\mathcal E$-components. Taking a totalization, we get a monomorphism
 \begin{displaymath}
  F \to \op{tot}^{\bigoplus}(E_{\bullet}). 
 \end{displaymath}
 By assumption, the factorization, $\op{tot}^{\bigoplus}(E_{\bullet})$, has components lying in $\mathcal E$.
 
 Thus, the natural functor,
 \begin{displaymath}
  \op{D}^{\op{co}}(\op{Fact} \mathcal E,w) \to \op{D}^{\op{co}}(\op{Fact} w)
 \end{displaymath}
 is essentially surjective. We next check fully-faithfulness. 
 
 For fully-faithfulness, since any bounded exact complex can be split into short exact sequences, it suffices to show that given a short exact sequence
 \begin{equation} \label{equation: fact SES}
  0 \to F_0 \to F_1 \to F_2 \to 0
 \end{equation}
 of factorizations in $\op{Fact}(w)$, there exists a factorization, $S \in \op{Acyc}(\mathcal E,w)$, that is isomorphic to the totalization, $T$, of \eqref{equation: fact SES} in $\op{D}^{\op{co}}(\op{Fact} w)$. 
 
 Using what we have already proven, we can find a factorization $E_{0,0}$ with components in $\mathcal E$ and a monomorphism 
 \begin{displaymath}
  F_0 \to E_{0,0}.
 \end{displaymath}
 Next choose a factorization $E_{1,0}$ with components in $\mathcal E$ and a monomorphism from the pushout
 \begin{displaymath}
  F_1 \oplus_{F_0} E_{0,0} \to E_{1,0}.
 \end{displaymath}
 Let $E_{2,0}$ be a factorization with components in $\mathcal E$ admitting a monomorphism from the cokernel of the map $F_1 \oplus_{F_0} E_{0,0} \to E_{1,0}$. And inductively for $n \geq 3$, let $E_{n,0}$ be a factorization with components in $\mathcal E$ admitting a monomorphism from cokernel of the map $E_{n-2,0} \to E_{n-1,0}$.
 There is a commutative diagram 
 \begin{center}
 \begin{tikzpicture}[description/.style={fill=white,inner sep=2pt}]
  \matrix (m) [matrix of math nodes, row sep=1em, column sep=3em, text height=1.5ex, text depth=0.25ex]
  {  0 & F_0 & F_1 & F_2 & 0 & \cdots \\
     0 & E_{0,0} & E_{1,0} & E_{2,0} & E_{3,0} & \cdots \\
  };
  \path[->,font=\scriptsize]
  (m-1-1) edge (m-1-2)
  (m-1-2) edge (m-1-3)
  (m-1-3) edge (m-1-4)
  (m-1-4) edge (m-1-5)
  (m-1-5) edge (m-1-6)
  (m-2-1) edge (m-2-2)
  (m-2-2) edge (m-2-3)
  (m-2-3) edge (m-2-4)
  (m-2-4) edge (m-2-5)
  (m-2-5) edge (m-2-6)
  (m-1-2) edge (m-2-2)
  (m-1-3) edge (m-2-3)
  (m-1-4) edge (m-2-4)
  (m-1-5) edge (m-2-5)
  ;
 \end{tikzpicture}
 \end{center}
 with the vertical morphisms being monomomorphism and the rows being exact. Set $E_{n,-1} = 0$. For $m \geq 1, n \geq 0$, let $E_{m,n}$ be a factorization with $\mathcal E$ components receiving a monomorphism from the pushout $E_{m-1,n} \oplus_{E_{m-1,n-1}} E_{m,n-1}$. We get an exact sequence of exact sequences
 \begin{center}
 \begin{tikzpicture}[description/.style={fill=white,inner sep=2pt}]
  \matrix (m) [matrix of math nodes, row sep=1em, column sep=3em, text height=1.5ex, text depth=0.25ex]
  {  & 0 & 0 & 0 &  &  \\
     0 & F_0 & F_1 & F_2 & 0 &  \\
     0 & E_{0,0} & E_{1,0} & E_{2,0} & E_{3,0} & \cdots \\
     & \vdots & \vdots & \vdots & \vdots &  \\
     0 & E_{0,s} & E_{1,s} & E_{2,s} & E_{3,s} & \cdots \\
     & \vdots & \vdots & \vdots & \vdots &  \\
  };
  \path[->,font=\scriptsize]
  (m-2-1) edge (m-2-2)
  (m-2-2) edge (m-2-3)
  (m-2-3) edge (m-2-4)
  (m-2-4) edge (m-2-5)
  
  (m-3-1) edge (m-3-2)
  (m-3-2) edge (m-3-3)
  (m-3-3) edge (m-3-4)
  (m-3-4) edge (m-3-5)
  (m-3-5) edge (m-3-6)
  
  (m-5-1) edge (m-5-2)
  (m-5-2) edge (m-5-3)
  (m-5-3) edge (m-5-4)
  (m-5-4) edge (m-5-5)
  (m-5-5) edge (m-5-6)
  
  (m-1-2) edge (m-2-2)
  (m-1-3) edge (m-2-3)
  (m-1-4) edge (m-2-4)
  
  (m-2-2) edge (m-3-2)
  (m-2-3) edge (m-3-3)
  (m-2-4) edge (m-3-4)
  (m-2-5) edge (m-3-5)
  
  (m-3-2) edge (m-4-2)
  (m-3-3) edge (m-4-3)
  (m-3-4) edge (m-4-4)
  (m-3-5) edge (m-4-5)
  
  (m-4-2) edge (m-5-2)
  (m-4-3) edge (m-5-3)
  (m-4-4) edge (m-5-4)
  (m-4-5) edge (m-5-5)
  
  (m-5-2) edge (m-6-2)
  (m-5-3) edge (m-6-3)
  (m-5-4) edge (m-6-4)
  (m-5-5) edge (m-6-5)
  ;
 \end{tikzpicture}
 \end{center}
 where each $E_{i,j}$ has components in $\mathcal E$. We take totalizations to get an exact sequence
 \begin{displaymath}
  0 \to \op{tot}(F_{\bullet}) \to \op{tot}^{\bigoplus}(E_{\bullet,0}) \to \cdots \to \op{tot}^{\bigoplus}(E_{\bullet,s}) \to \cdots.
 \end{displaymath}
 By Lemma~\ref{lemma: infinite tots are acyclic}, each $\op{tot}^{\bigoplus}(E_{\bullet,s})$ is co-acyclic. Thus, the totalization $\op{tot}^{\bigoplus}(\op{tot}^{\bigoplus}(E_{\bullet,\bullet}))$ of 
 \begin{displaymath}
  \op{tot}^{\bigoplus}(E_{\bullet,0}) \to \cdots \to \op{tot}^{\bigoplus}(E_{\bullet,s}) \to \cdots
 \end{displaymath}
 lies in $\op{Co-acyc}(\mathcal E,w)$ and is isomorphic to $\op{tot}(F_{\bullet})$ in $\op{D}^{\op{co}}(\op{Fact} w)$.
\end{proof}

There is also the dual statement which we record separately.

\begin{proposition} \label{proposition: adapted contra-replacement}
 Assume that $\mathcal A$ has small products. Let $\mathcal E$ be an additive full subcategory of $\mathcal A$ preserved by $\Phi$ and satisfying the following conditions:
 \begin{itemize}
  \item $\mathcal E$ is closed under products. 
  \item For any object $A \in \mathcal A$, there exists a epimorphism
  \begin{displaymath}
   E \to A
  \end{displaymath}
  with $E$ an object of $\mathcal E$.
 \end{itemize}
 Then, the composition
 \begin{displaymath}
  \op{K}(\op{Fact} \mathcal E,w) \to \op{K}(\op{Fact}(\mathcal A,w)) \to \op{D}^{\op{ctr}}(\op{Fact} w)
 \end{displaymath}
 induces an equivalence
 \begin{displaymath}
  \op{Q}_{\mathcal E} : \op{D}^{\op{ctr}}(\op{Fact} \mathcal E,w) \to \op{D}^{\op{ctr}}(\op{Fact} w).
 \end{displaymath}
\end{proposition}

\begin{proof}
 This proof is completely analogous to that of Proposition~\ref{proposition: adapted co-replacement} and is therefore suppressed.
\end{proof}

Finally, modifying the assumptions slightly, we have an analogous statement for absolute derived categories.

\begin{proposition} \label{proposition: adapted abs-replacement}
 Let $\mathcal E$ be an additive full subcategory of $\mathcal A$ preserved by $\Phi$ and satisfying the following conditions: 
 \begin{itemize}
  \item For any object $A \in \mathcal A$, there exists a monomorphism
  \begin{displaymath}
   A \to E
  \end{displaymath}
  with $E$ an object of $\mathcal E$.
  \item There exists an $N$ such that for any exact sequence
  \begin{displaymath}
    0 \to A \to E_0 \to \cdots \to E_{n-1} \to E_{n},
  \end{displaymath}
  with each $E_i$ lying in $\mathcal E$, the cokernel of the morphism $E_{n-1} \to E_{n}$ lies in $\mathcal E$ whenever $|n| \geq N$.
 \end{itemize}
 Or satisfying the following dual conditions:
 \begin{itemize}
  \item For any object $A \in \mathcal A$, there exists a epimorphism
  \begin{displaymath}
   E \to A
  \end{displaymath}
  with $E$ an object of $\mathcal E$.
  \item There exists an $N$ such that for any exact sequence
  \begin{displaymath}
    E_{n} \to E_{n+1} \to \cdots \to E_0 \to A \to 0,
  \end{displaymath}
  with each $E_i$ lying in $\mathcal E$, the kernel of the morphism $E_{n} \to E_{n+1}$ lies in $\mathcal E$ whenever $n \geq N$.
 \end{itemize}
 
 Under either set of assumptions, the composition
 \begin{displaymath}
  \op{K}(\op{Fact} \mathcal E,w) \to \op{K}(\op{Fact}(\mathcal A,w)) \to \op{D}^{\op{abs}}(\op{Fact} w)
 \end{displaymath}
 induces an equivalence
 \begin{displaymath}
  \op{Q}_{\mathcal E} : \op{D}^{\op{abs}}(\op{Fact} \mathcal E,w) \to \op{D}^{\op{abs}}(\op{Fact} w).
 \end{displaymath}
\end{proposition}

\begin{proof}
 The proof proceeds in a fashion completely analogous to that of Proposition~\ref{proposition: adapted co-replacement}, or Proposition~\ref{proposition: adapted contra-replacement} with the dual set of assumptions, with the exception that new hypothesis allows one to deal with a bounded bicomplex, obviating the need for coproducts or products in the totalization. 
\end{proof}

\begin{remark}
 In Section~\ref{section: resolutions}, we will see another method for producing injective or projective resolutions of factorizations. These will provide more control than those appearing in the arguments of the proof of Proposition~\ref{proposition: adapted co-replacement}.
\end{remark}

Following the analogy with derived categories of Abelian categories, one can realize the various derived categories of factorizations as homotopy categories of factorizations with injective or projective components.

\begin{lemma} \label{lemma: injective, projective fact orthogonal}
 Let $I$ be an object of $\op{Fact}(w)$ with $I^{-1},I^{0}$ injective objects of $\mathcal A$. Let $C$ be a co-acyclic factorization. Then,
 \begin{displaymath}
  \op{Hom}_{\op{K}(\op{Fact} w)}(C,I) = 0.
 \end{displaymath}
 
 Let $P$ be an object of $\op{Fact}(w)$ with $P^{-1},P^{0}$ projective objects of $\mathcal A$. Let $C$ be a contra-acyclic factorization. Then,
 \begin{displaymath}
  \op{Hom}_{\op{K}(\op{Fact} w)}(P,C) = 0.
 \end{displaymath}
\end{lemma}

\begin{proof}
 If $C_s, s \in S$ is a collection of objects left orthogonal to $I$, then $\bigoplus_{s\in S} C_s$ is also left orthogonal to $I$. We can reduce to checking that $I$ is right orthogonal to totalizations of exact sequences. Any exact sequence is an iterated sequence of totalizations of short exact sequences. Thus, it suffices to check that $I$ is left orthogonal to totalizations of short exact sequences. 
 
 Take a short exact sequence of factorizations, 
 \begin{displaymath}
  0 \to E_1 \overset{g_1}{\to} E_2 \overset{g_2}{\to} E_3 \to 0.
 \end{displaymath}
 Let $C$ be the totalization of this short exact sequence. By definition, there is a triangle,
 \begin{displaymath}
  E_1[1] \overset{h}{\to} C(g_2) \to C \to E_1[2],
 \end{displaymath}
 in $\op{K}(\op{Fact} w)$. Therefore, there is a long exact sequence,
 \begin{gather*}
  \cdots \to \op{Hom}_{\op{K}(\op{Fact} w)}(C(g_2)[i+1],I) \to \op{Hom}_{\op{K}(\op{Fact} w)}(E_1[i+2],I) \to \op{Hom}_{\op{K}(\op{Fact} w)}(C[i],I) \\ \to \op{Hom}_{\op{K}(\op{Fact} w)}(C(g_2)[i],I) \to \op{Hom}_{\op{K}(\op{Fact} w)}(E_1[i+1],I) \to \cdots
 \end{gather*}
 Showing that
 \begin{displaymath}
  \op{Hom}_{\op{K}(\op{Fact} w)}(C[i],I) = 0
 \end{displaymath}
 for all $i$ is equivalent to showing that the maps,
 \begin{displaymath}
  \op{Hom}_{\op{K}(\op{Fact} w)}(C(g_2)[i],I) \to \op{Hom}_{\op{K}(\op{Fact} w)}(E_1[i+1],I),
 \end{displaymath}
 are isomorphisms for all $i$.

 There is a commutative diagram,
 \begin{center}
 \begin{tikzpicture}[description/.style={fill=white,inner sep=2pt}]
 \matrix (m) [matrix of math nodes, row sep=2em, column sep=2em, text height=1.5ex, text depth=0.25ex]
 {  E_1 & E_2 & E_3  \\
    C(g_2)[-1] & E_2 & E_3 \\ };
 \path[->,font=\scriptsize]
  (m-1-1) edge (m-1-2) 
  (m-1-2) edge node[above] {$g_2$}(m-1-3)

  (m-2-1) edge (m-2-2) 
  (m-2-2) edge node[above] {$g_2$}(m-2-3)
  
  (m-1-1) edge node[left] {$h[-1]$} (m-2-1)
  (m-1-2) edge node[left] {$=$} (m-2-2)
  (m-1-3) edge node[left] {$=$} (m-2-3)
 ;
 \end{tikzpicture} 
 \end{center}
 Apply $\textbf{Hom}_w^*(\bullet, I)$ to this diagram to get a commutative diagram of complexes,
 \begin{center}
 \begin{tikzpicture}[description/.style={fill=white,inner sep=2pt}]
 \matrix (m) [matrix of math nodes, row sep=2em, column sep=2em, text height=1.5ex, text depth=0.25ex]
 {  \textbf{Hom}_w^*(E_3, I) & \textbf{Hom}_w^*(E_2, I) & \textbf{Hom}_w^*(C(g_2)[-1], I) \\
    \textbf{Hom}_w^*(E_3, I) & \textbf{Hom}_w^*(E_2, I) & \textbf{Hom}_w^*(E_1, I)  \\
     };
 \path[->,font=\scriptsize]
  (m-1-1) edge (m-1-2) 
  (m-1-2) edge (m-1-3)

  (m-2-1) edge (m-2-2) 
  (m-2-2) edge (m-2-3)
  
  (m-1-1) edge node[left] {$=$} (m-2-1)
  (m-1-2) edge node[left] {$=$} (m-2-2)
  (m-1-3) edge node[left] {$h[-1]$} (m-2-3)
 ;
 \end{tikzpicture} 
 \end{center}
 
 Since $I$ has injective components, the sequence,
 \begin{displaymath}
  0 \to \textbf{Hom}_w^*(E_3, I) \to \textbf{Hom}_w^*(E_2, I) \to \textbf{Hom}_w^*(E_1, I) \to 0,
 \end{displaymath}
 is an exact sequence of complexes.

 Taking cohomology of all the complexes in the diagram above induces a morphism of long exact sequences,
 \begin{center}
 \begin{tikzpicture}[description/.style={fill=white,inner sep=2pt},scale=0.9, every node/.style={scale=0.9}]
 \matrix (m) [matrix of math nodes, row sep=2em, column sep=1.5em, text height=1.5ex, text depth=0.25ex]
 {  \cdots & \op{Hom}_{\op{K}(\op{Fact} w)}(E_3[i], I) & \op{Hom}_{\op{K}(\op{Fact} w)}(E_2[i], I) & \op{Hom}_{\op{K}(\op{Fact} w)}(C(g_2)[i-1], I) & \cdots  \\
    \cdots & \op{Hom}_{\op{K}(\op{Fact} w)}(E_3[i], I) & \op{Hom}_{\op{K}(\op{Fact} w)}(E_2[i], I) & \op{Hom}_{\op{K}(\op{Fact} w)}(E_1[i], I)  & \cdots \\
    };
 \path[->,font=\scriptsize]
  (m-1-1) edge (m-1-2) 
  (m-1-2) edge (m-1-3)
  (m-1-3) edge (m-1-4)
  (m-1-4) edge (m-1-5)

  (m-2-1) edge (m-2-2) 
  (m-2-2) edge (m-2-3)
  (m-2-3) edge (m-2-4)
  (m-2-4) edge (m-2-5)
  
  (m-1-2) edge node[left] {$=$} (m-2-2)
  (m-1-3) edge node[left] {$=$} (m-2-3)
  (m-1-4) edge node[left] {$h[i-1]$} (m-2-4)
 ;
 \end{tikzpicture} 
 \end{center}
 From the $5$-lemma, we can conclude that $h[i]$ is an isomorphism for all $i$.
 
 The proof for contra-acyclic and projective factorizations is completely analogous.
\end{proof}

In the case of factorizations with injective or projective components, we do not need to take any further quotients.

\begin{corollary} \label{corollary: homotopy category of inj/proj is derived cat}
 If $\mathcal A$ has enough injectives and coproducts of injectives are injective, then the composition
 \begin{displaymath}
  \op{Q}_{\op{inj}}: \op{K}(\op{Fact} \mathcal Inj \ w) \to \op{K}(\op{Fact} w) \to \dcoFact{w}
 \end{displaymath}
 is an equivalence.
 
 If $\mathcal A$ has enough projectives and products of projectives are projective, then the composition
 \begin{displaymath}
  \op{Q}_{\op{proj}}: \op{K}(\op{Fact} \mathcal Proj \ w) \to \op{K}(\op{Fact} w) \to \dconFact{w}
 \end{displaymath}
 is an equivalence. 
 
 If $\mathcal A$ has finite injective dimension, then the composition
 \begin{displaymath}
  \op{Q}_{\op{inj}}: \op{K}(\op{Fact} \mathcal Inj \ w) \to \op{K}(\op{Fact} w) \to \dabsFact{w}
 \end{displaymath}
 is an equivalence.
 
 If $\mathcal A$ has finite projective dimension, then the composition
 \begin{displaymath}
  \op{Q}_{\op{proj}}: \op{K}(\op{Fact} \mathcal Proj \ w) \to \op{K}(\op{Fact} w) \to \dabsFact{w}
 \end{displaymath}
 is an equivalence.
\end{corollary}

\begin{proof}
 Lemma~\ref{lemma: injective, projective fact orthogonal} shows that any co-acyclic or totally acyclic factorization with injective components is zero in the homotopy category and any contra-acyclic or totally-acyclic factorization with projective components is zero in the homotopy category. Then Proposition~\ref{proposition: adapted co-replacement} gives the first statement, Proposition~\ref{proposition: adapted contra-replacement} gives the second, and Proposition~\ref{proposition: adapted abs-replacement} gives the last two.
\end{proof}

Finally, we record a fact that allows one to reduce some arguments to factorizations with zero component morphisms.

\begin{lemma}
For any factorization $E = (E^{-1}, E^0, \phi_{E}^{-1}, \phi_{E}^0)$, there is an exact sequence in $\op{Fact}(w)$,
\[
0 \to \op{ker} \phi_{E}^{0} \overset{f}{\longrightarrow} (E^{-1}, E^{-1}, w_{\Phi^{-1}(E^{-1})}, \op{1}_{E^{-1}}) \overset{g}{\longrightarrow} E \overset{h}{\longrightarrow} \op{coker} \phi_{E}^{0} \to 0.
\]
This gives rise to an exact triangle in $\dabsFact{w}$,
\[
\op{coker} \phi_{E}^{0} \to \op{ker} \phi_{E}^{0}[2] \to E.
\]
\label{lem: breakemup}
\end{lemma}

\begin{proof}
 The components of the morphisms $f,g,h$ are given by 
 \[
  f^{-1} = 0,  f^0  = i, g^{-1} = 1_{E^{-1}},  g^0 = \phi_{E}^{-1}, h^{-1} = 0 , h^0 = \pi
 \]
 where is $i: \op{ker} \phi_E^{0} \to E^{-1}$ is the inclusion and $\pi: E^{0} \to \op{coker} \phi_E^{0}$ is the projection. It is straightforward to see that the sequeneces associated to each component are exact.
\end{proof}

\begin{definition}
 Let $\mathcal T$ be a triangulated category. A subcategory $\mathcal S$ is said to \newterm{triangularly generate} $\mathcal T$ if the smallest triangulated subcategory $\mathcal T$ containing $\mathcal S$ is $\mathcal T$. 
\end{definition}

\begin{remark}
 The usual notion of generation includes closure under formation of summands \cite{BV}. Our language reflects the fact that only formation of cones is allowed.
\end{remark}

\begin{corollary}
 The categories $\dabsFact{w}, \dcoFact{w}, \dconFact{w}$ are each triangularly generated by factorizations of the form $(0, A, 0, 0)$ for $A \in \mathcal A$.
\label{cor: single object generation}
\end{corollary}

\begin{proof}
 This follows immediately from the exact triangle in Lemma~\ref{lem: breakemup}.
\end{proof}

\begin{remark}
 In fact, $\dabsFact{w}$ is strongly triangularly generated by objects of the form $(0, A, 0, 0)$ for $A \in \mathcal A$ as we only need to take a single cone. See \cite{BV} for a definition of strong generation.
\end{remark}

\section{Constructions of resolutions} \label{section: resolutions}

In this section, we provide a useful method of replacing a factorization by a co-quasi-isomorphic factorization of injectives or by a contra-quasi-isomorphic factorization of projectives. We saw a few simple consequences of the existence of such replacements at the end of Section \ref{section: basics}. In Section \ref{section: applications}, we will present some more computationally-useful applications.

We first analyze a way to construct factorizations starting from complexes over $\mathcal A$. Assume we have two complexes of objects of $\mathcal A$
\begin{center}
 \begin{tikzpicture}[description/.style={fill=white,inner sep=2pt}]
 \matrix (m) [matrix of math nodes, row sep=1em, column sep=2.5em, text height=1.5ex, text depth=0.25ex]
 {  \cdots & A^{-1}_{-1} & A^{-1}_0 & A^{-1}_1 & A^{-1}_2 & \cdots \\ 
    \cdots & A^{0}_{-1} & A_0^0 & A_1^0 & A^0_2 & \cdots. \\
    };
 \path[->,font=\scriptsize]

  (m-1-1) edge node[above] {$d_{-1}^{-1}$}(m-1-2)
  (m-1-2) edge node[above] {$d_0^{-1}$}(m-1-3)
  (m-1-3) edge node[above] {$d_1^{-1}$}(m-1-4)
  (m-1-4) edge node[above] {$d_2^{-1}$}(m-1-5)
  (m-1-5) edge node[above] {$d_3^{-1}$}(m-1-6)
  
  (m-2-1) edge node[above] {$d_{-1}^0$}(m-2-2)
  (m-2-2) edge node[above] {$d_0^0$}(m-2-3)
  (m-2-3) edge node[above] {$d_1^0$}(m-2-4)
  (m-2-4) edge node[above] {$d_2^0$}(m-2-5)
  (m-2-5) edge node[above] {$d_3^0$}(m-2-6)
 ;
 \end{tikzpicture} 
\end{center}

If the either the complexes is infinite, we assume that $\mathcal A$ has small coproducts or small products. Define the following two objects of $\mathcal A$ by combining even and odd components of the two complexes:
\begin{align*}
 \op{tot}^{\bigoplus}(A_{\bullet})^{-1} & := \bigoplus_{2l} \Phi^{-l}(A^{-1}_{2l}) \oplus \bigoplus_{2l+1} \Phi^{-l-1}(A^0_{2l+1})  \\
 \op{tot}^{\bigoplus}(A_{\bullet})^0 & := \bigoplus_{2l} \Phi^{-l}(A^0_{2l}) \oplus \bigoplus_{2l+1} \Phi^{-l}(A^{-1}_{2l+1}). 
\end{align*}
Similarly, set
\begin{align*}
 \op{tot}^{\prod}(A_{\bullet})^{-1} & := \prod_{2l} \Phi^{-l}(A^{-1}_{2l}) \times \prod_{2l+1} \Phi^{-l-1}(A^0_{2l+1})  \\
 \op{tot}^{\prod}(A_{\bullet})^0 & := \prod_{2l} \Phi^{-l}(A^0_{2l}) \times \prod_{2l+1} \Phi^{-l}(A^{-1}_{2l+1}).
\end{align*}

\begin{definition}
 We say the two complexes $(A^{-1}_{\bullet},A^{0}_{\bullet})$ are \newterm{$\bigoplus$-foldable} if there exists a factorization $A = (\op{tot}^{\bigoplus}(A_{\bullet})^{-1},\op{tot}^{\bigoplus}(A_{\bullet})^{0},\phi^{-1}_A,\phi^0_A)$ such that
 \begin{displaymath}
  \phi^{-1}_{p,q} = \phi^{0}_{p,q} = 0 \text{ for } q > p+1
 \end{displaymath}
 and 
 \begin{align*}
   \phi^{-1}_{2l+1,2l+2} & =    \Phi^{-l-1}(d^{-1}_{2l+2})  : \Phi^{-l-1}(A^{-1}_{2l+1}) \to \Phi^{-l-1}(A^{-1}_{2l+2}) \\
   \phi^{-1}_{2l,2l+1} & =    -\Phi^{-l-1}(d^{0}_{2l+1}) : \Phi^{-l-1}(A^{0}_{2l}) \to \Phi^{-l-1}(A^{0}_{2l+1}) \\
   \phi^{0}_{2l+1,2l+2}  & =   -\Phi^{-l-1}(d^{0}_{2l+2})  : \Phi^{-l-1}(A^{0}_{2l+1}) \to \Phi^{-l-1}(A^{0}_{2l+2}) \\
   \phi^{0}_{2l,2l+1} & =  \Phi^{-l}(d^{-1}_{2l+1}) : \Phi^{-l}(A^{-1}_{2l}) \to \Phi^{-l}(A^{-1}_{2l+1}),
  \end{align*}
 where
 \begin{align*}
  \phi^{-1}_{2l+1,2j+1} & : \Phi^{-l-1}(A^{-1}_{2l+1}) \to \Phi^{-1}(A^{0}) \overset{\phi_A^{-1}}{\to} A^{-1} \to \Phi^{-j-1}(A^{0}_{2j+1}) \\
  \phi^{-1}_{2l+1,2j} & : \Phi^{-l-1}(A^{-1}_{2l+1}) \to \Phi^{-1}(A^{0}) \overset{\phi_A^{-1}}{\to} A^{-1} \to \Phi^{-j}(A^{-1}_{2j}) \\
  \phi^{-1}_{2l,2j+1} & : \Phi^{-l-1}(A^{0}_{2l}) \to \Phi^{-1}(A^{0}) \overset{\phi_A^{-1}}{\to} A^{-1} \to \Phi^{-j-1}(A^{0}_{2j+1}) \\
  \phi^{-1}_{2l,2j} & : \Phi^{-l-1}(A^{0}_{2l}) \to \Phi^{-1}(A^{0}) \overset{\phi_A^{-1}}{\to} A^{-1} \to \Phi^{-j}(A^{-1}_{2j}) \\
 \end{align*}
 and
 \begin{align*}
  \phi^{0}_{2l+1,2j+1} & : \Phi^{-l-1}(A^{0}_{2l+1}) \to A^{-1} \overset{\phi_A^{0}}{\to} A^{0} \to \Phi^{-j}(A^{-1}_{2j+1}) \\
  \phi^{0}_{2l+1,2j} & : \Phi^{-l-1}(A^{0}_{2l+1}) \to A^{-1} \overset{\phi_A^{0}}{\to} A^{0} \to \Phi^{-j}(A^{0}_{2j}) \\
  \phi^{0}_{2l,2j+1} & : \Phi^{-l}(A^{-1}_{2l}) \to A^{-1} \overset{\phi_A^{0}}{\to} A^{0} \to \Phi^{-j}(A^{-1}_{2j+1}) \\
  \phi^{0}_{2l,2j} & : \Phi^{-l}(A^{-1}_{2l}) \to A^{-1} \overset{\phi_A^{0}}{\to} A^{0} \to \Phi^{-j}(A^{0}_{2j}).
 \end{align*}
 Any such factorization $A$ will be called a \newterm{$\bigoplus$-folding} of $(A^{-1}_{\bullet},A^{0}_{\bullet})$, and, in general, a \newterm{$\bigoplus$-folded factorization}.
 
  We say the two complexes $(A^{-1}_{\bullet},A^{0}_{\bullet})$ are \newterm{$\prod$-foldable} if there exists a factorization $A = (\op{tot}^{\prod}(A_{\bullet})^{-1},\op{tot}^{\prod}(A_{\bullet})^{0},\phi^{-1}_A,\phi^0_A)$ such that
 \begin{displaymath}
  \phi^{-1}_{p,q} = \phi^{0}_{p,q} = 0 \text{ for } q > p+1
 \end{displaymath}
 and 
 \begin{align*}
   \phi^{-1}_{2l+1,2l+2} & =    \Phi^{-l-1}(d^{-1}_{2l+2})  : \Phi^{-l-1}(A^{-1}_{2l+1}) \to \Phi^{-l-1}(A^{-1}_{2l+2}) \\
   \phi^{-1}_{2l,2l+1} & =    -\Phi^{-l-1}(d^{0}_{2l+1}) : \Phi^{-l-1}(A^{0}_{2l}) \to \Phi^{-l-1}(A^{0}_{2l+1}) \\
   \phi^{0}_{2l+1,2l+2}  & =   -\Phi^{-l-1}(d^{0}_{2l+2})  : \Phi^{-l-1}(A^{0}_{2l+1}) \to \Phi^{-l-1}(A^{0}_{2l+2}) \\
   \phi^{0}_{2l,2l+1} & =  \Phi^{-l}(d^{-1}_{2l+1}) : \Phi^{-l}(A^{-1}_{2l}) \to \Phi^{-l}(A^{-1}_{2l+1}).
  \end{align*}
 Any such factorization $A$ will be called a \newterm{$\prod$-folding} of $(A^{-1}_{\bullet},A^{0}_{\bullet})$, and, in general, a \newterm{$\prod$-folded factorization}.
 
 We shall often drop the $\bigoplus$ or $\prod$ if the context allows.
\end{definition}

\begin{remark} \label{remark: Kanye}
 As we will see, these definitions are built to allow us to linearize computations using folded factorizations. If both $A^{-1}_{\bullet}$ and $A^0_{\bullet}$ are both bounded below or both bounded above and $A$ is a folding, requiring that 
 \begin{align*}
  \phi^0_A \circ \phi^{-1}_A & = w_{\Phi^{-1}(A^0)} \\
  \Phi(\phi^{-1}_A)\circ \phi^0_A & = w_{A^{-1}}
 \end{align*}
 is equivalent to requiring the following identities of the components of $\phi^{-1}_A$ and $\phi^0_A$:
 \begin{displaymath}
  \sum_{t \in \Z} \phi^0_{t,q} \circ \phi^{-1}_{p,t} =
  \begin{cases}
   0 & p \not = q \\
   w & p=q
  \end{cases}
 \end{displaymath}
 and
 \begin{displaymath}
  \sum_{t \in \Z} \Phi(\phi^{-1}_{t,q}) \circ \phi^{0}_{p,t} =
  \begin{cases} 
   0 & p \not = q \\
   w & p=q.
  \end{cases}
 \end{displaymath}
 As only finitely many terms in these sums will be nonzero, these equations are completely unambiguous.
\end{remark}

\begin{remark} \label{remark: Jay-Z}
 Note that if $A$ folds $(A^{-1}_{\bullet},A^0_{\bullet})$ then $A[1]$ folds $(A^{-1}[1]_{\bullet},A^0[1]_{\bullet})$.
\end{remark}

\begin{lemma} \label{lemma: acyclic folds to acyclic}
 If $A$ is a factorization folding a pair of bounded exact complexes $(A^{-1}_{\bullet},A^0_{\bullet})$, then $A$ is totally-acyclic. 
 
 Assume that $\mathcal A$ has small coproducts. If $A$ is a factorization $\bigoplus$-folding a pair of bounded below exact complexes $(A^{-1}_{\bullet},A^0_{\bullet})$, then $A$ is co-acyclic. 
 
 Assume that $\mathcal A$ has small products. If $A$ is a factorization $\prod$-folding a pair of bounded below exact complexes $(A^{-1}_{\bullet},A^0_{\bullet})$, then $A$ is contra-acyclic. 
\end{lemma}

\begin{proof}
 For those used to derived categories, the idea is quite simple; the cone of the morphism behaves like the sum of the two good truncations of the resolutions, hence, like a complex with no cohomology.  Morally, this complex is then split into short exact sequences.  In the language of factorizations this amounts to constructing the complex as a colimit of totalizations, which is finite when the resolutions are finite.
 
 After replacing $A^{-1}_{\bullet}$ and $A^0_{\bullet}$ by a common appropriate shift, we may assume that both complexes vanish in negative degrees. Let $C_0^{-1}=A_0^{-1}$ and let $C_{j}^{-1}$ be the cokernel of $d^{-1}_j: A^{-1}_{j-1} \to A^{-1}_{j}$. Let $C_{0}^0 = A^0_0$ and let $C^0_j$ be the cokernel of $d^0_j: A^0_{j-1} \to A^0_{j}$. From exactness, $C_{j}^0$ is the kernel of $d^0_{j+1}$ and $C_{j}^{-1}$ is the kernel of $d^{-1}_{j+1}$ and we have exact sequences,
 \begin{align*}
  0 \to C_{j-1}^{-1} \to A_{j-1}^{-1} \overset{d^{-1}_j}{\to} A^{-1}_j \to C_{j}^{-1} \to 0 \\
  0 \to C_{j-1}^{0} \to A_{j-1}^{0} \overset{d^{0}_j}{\to} A^{0}_j \to C_{j}^{0} \to 0.
 \end{align*}
 
 Consider the subfactorization, $\tau_{\leq j} A$, of $A$ given by restricting the components to their good truncations. The factorization, $\tau_{\leq j} A$, has components,
 \begin{align*}
  \tau_{\leq j} A^0 & = \bigoplus_{0 \leq 2l < j} \Phi^{-l}(A^0_{2l}) \oplus \bigoplus_{0 \leq 2l+1 < j} \Phi^{-l}(A^{-1}_{2l+1}) \oplus \begin{cases} \Phi^{-t}(C_{j}^{0}) & j = 2t \\ \Phi^{-t}(C_{j}^{-1}) & j = 2t+1, \end{cases} \\
  \tau_{\leq j} A^{-1} & = \bigoplus_{0 \leq 2l+1 < j} \Phi^{-l-1}(A^0_{2l+1}) \oplus \bigoplus_{0 \leq 2l < j} \Phi^{-l}(A^{-1}_{2l}) \oplus \begin{cases}  \Phi^{-t}(C_{j}^{-1}) & j = 2t \\ \Phi^{-t-1}(C_{j}^{0}) & j = 2t+1, \end{cases}
 \end{align*}
 and morphisms between components induced by those from $A$ using the inclusion $C^j_p \to A^j_{p+1}$. Note that this is a well-defined factorization since $\phi^j_{p,q} = 0$ for $q > p+1$ and $d^j_{p+1}$ vanishes on $C^j_p$. Let $S_i$ denote the factorization with components,
 \begin{align*}
  S_j^0 & = \begin{cases} \Phi^{-t+1}(C_j^{-1}) \oplus \Phi^{-t}(C_j^0) & j = 2t \\ \Phi^{-t}(C_j^{0}) \oplus \Phi^{-t}(C_j^{-1}) & j = 2t + 1 \end{cases} \\
  S_j^{-1} & = \begin{cases} \Phi^{-t}(C_j^{0}) \oplus \Phi^{-t}(C_j^{-1}) & j = 2t \\ \Phi^{-t}(C_j^{-1}) \oplus \Phi^{-t-1}(C_j^{0}) & j = 2t + 1 \end{cases}
 \end{align*}
 and morphisms
 \begin{align*}
  \phi_{S_j}^0 & = \begin{cases}\begin{pmatrix} 0 & w_{\Phi^{-t}(C_j^{-1})} \\  \op{id}_{\Phi^{-t}(C_j^0)} & 0 \end{pmatrix} & j = 2l \\ \begin{pmatrix} 0 & w_{\Phi^{-t-1}(C_j^{0})} \\  \op{id}_{\Phi^{-t}(C_j^{-1})} & 0 \end{pmatrix} & j = 2l + 1 \end{cases} \\
  \phi_{S_j}^{-1} & = \begin{cases} \begin{pmatrix} 0 & w_{\Phi^{-t-1}(C_j^{0})} \\  \op{id}_{\Phi^{-t}(C_j^{-1})} & 0 \end{pmatrix} & j = 2l \\ \begin{pmatrix} 0 & w_{\Phi^{-t-1}(C_j^{-1})} \\  \op{id}_{\Phi^{-t-1}(C_j^0)} & 0 \end{pmatrix} & j = 2l + 1 \end{cases}
 \end{align*}
 Note that $S_j$ is manifestly a null-homotopic factorization. There are short exact sequences,
 \begin{displaymath}
  0 \to \tau_{\leq j} A \to \tau_{\leq j+1} A \to S_{j+1} \to 0,
 \end{displaymath}
 of factorizations. 
 
 Thus, in $\dabsFact{w}$, $\tau_{\leq j} A$ and $\tau_{\leq j+1} A$ are isomorphic for $j \geq 0$. If the resolutions are finite, we see that, since $\tau_{\leq j} A = 0$ for $j >> 0$, $A$ is totally-acyclic.

 In general, the colimit of these morphisms is isomorphic to $A$. As we can write the colimit via the short exact sequence,
 \begin{displaymath}
  0 \to \bigoplus_{j \geq 0} \tau_{\leq j} A \to \bigoplus_{j \geq 0} \tau_{\leq j} A \to \op{colim} \tau_{\leq j} A = A \to 0,
 \end{displaymath}
 we see that $A$ is co-acyclic in general. 
 
 The argument in the situation where $\mathcal A$ has small products is analogous and omitted. 
\end{proof}

\begin{remark} 
 See also \cite[Section 3.2]{Becker}
\end{remark}

\begin{lemma} \label{lemma: cones over foldings are folded}
 Let $A$ be a $\bigoplus$-folding (respectively $\coprod$-folding) of $(A^{-1}_{\bullet},A^0_{\bullet})$ and $B$ be a $\bigoplus$-folding  (respectively $\coprod$-folding) of $(B^{-1}_{\bullet},B^0_{\bullet})$. Let $\eta: A \to B$ be a morphism of factorizations. Let 
 \begin{align*}
  \eta_{2l,2j}^{-1} & : \Phi^{-l}(A_{2l}^{-1}) \to \Phi^{-j}(B_{2j}^{-1}) \\
  \eta_{2l,2j+1}^{-1} & : \Phi^{-l}(A_{2l}^{-1}) \to \Phi^{-j-1}(B_{2j+1}^{0}) \\
  \eta_{2l+1,2j}^{-1} & : \Phi^{-l-1}(A_{2l+1}^{0}) \to \Phi^{-j}(B_{2j}^{-1}) \\
  \eta_{2l+1,2j+1}^{-1} & : \Phi^{-l-1}(A_{2l+1}^{0}) \to \Phi^{-j-1}(B_{2j+1}^{0}) \\
  \eta_{2l,2j}^{0} & : \Phi^{-l}(A_{2l}^{0}) \to \Phi^{-j}(B_{2j}^{0}) \\
  \eta_{2l,2j+1}^{0} & : \Phi^{-l}(A_{2l}^{0}) \to \Phi^{-j}(B_{2j+1}^{-1}) \\
  \eta_{2l+1,2j}^{0} & : \Phi^{-l}(A_{2l+1}^{-1}) \to \Phi^{-j}(B_{2j}^{0}) \\
  \eta_{2l+1,2j+1}^{0} & : \Phi^{-l}(A_{2l+1}^{-1}) \to \Phi^{-j}(B_{2j+1}^{-1})
 \end{align*}
 be the morphisms on the summands of the components of $A$ and $B$ determined by $\eta$. Assume that 
 \begin{equation} \label{equation: vanishing condition}
  \eta_{p,q}^{-1},\eta_{p,q}^{0} = 0 \text{ for } q > p
 \end{equation}
 and 
 \begin{align*}
  \tilde{\eta}^{-1} & := \Phi^p(\eta^{-1}_{2p,2p}),\Phi^p(\eta^0_{2p+1,2p+1}): A^{-1}_{\bullet} \to B^{-1}_{\bullet} \\
  \tilde{\eta}^0 & := \Phi^p(\eta^{0}_{2p,2p}),\Phi^{p+1}(\eta^{-1}_{2p+1,2p+1}): A^{0}_{\bullet} \to B^{0}_{\bullet}
 \end{align*}
 are chain maps. Then, the cone over $\eta$, $\op{Cone}(\eta)$, folds the cones over $\tilde{\eta}^{-1},\tilde{\eta}^0$. 
\end{lemma}

\begin{proof}
 It is clear that the components of $\op{Cone}(\eta)$ are of the correct form to fold the cones over $\tilde{\eta}^{-1},\tilde{\eta}^0$. We check the conditions on the morphisms. The morphisms in the factorization $\op{Cone}(\eta)$ are given by
 \begin{displaymath}
  \phi^{-1}_{\op{Cone}(\eta)} = \begin{pmatrix} \phi^{-1}_{A[1]} & 0 \\ \eta^{-1} & \phi^{-1}_B \end{pmatrix}, \phi^{0}_{\op{Cone}(\eta)} = \begin{pmatrix} \phi^{0}_{A[1]} & 0 \\ \eta^0 & \phi^{0}_B \end{pmatrix}
 \end{displaymath}
 
 The vanishing condition, Equation~\eqref{equation: vanishing condition}, together with the fact that $A$ and $B$ are $\bigoplus$-foldable  (respectively $\coprod$-foldable) implies that the terms $\phi_{p,q}^{-1},\phi^0_{p,q}$ vanish for $q > p+1$ in $\op{Cone}(\eta)$ while 
 \begin{displaymath}
  \phi^{-1}_{p,p+1} = \begin{pmatrix} (\phi^{-1}_{A[1]})_{p,p+1} & 0 \\ \eta^{-1}_{p,p} & (\phi^{-1}_B)_{p,p+1} \end{pmatrix}, \phi^{0}_{p,p+1} = \begin{pmatrix} (\phi^{0}_{A[1]})_{p,p+1} & 0 \\ \eta^0_{p,p} & (\phi^{0}_B)_{p,p+1} \end{pmatrix}.
 \end{displaymath}
 We see that this is of the appropriate form.
\end{proof}

Now we assume that $\mathcal A$ has enough injectives. Let $E$ be a factorization. Choose injective resolutions 
\begin{center}
 \begin{tikzpicture}[description/.style={fill=white,inner sep=2pt}]
 \matrix (m) [matrix of math nodes, row sep=1em, column sep=2.5em, text height=1.5ex, text depth=0.25ex]
 {  0 & E^{-1} & I^{-1}_0 & I^{-1}_1 & \cdots \\ 
    0 & E^0 & I_0^0 & I_1^0 & \cdots. \\
    };
 \path[->,font=\scriptsize]

  (m-1-1) edge (m-1-2) 
  (m-1-2) edge node[above] {$d_0^{-1}$}(m-1-3)
  (m-1-3) edge node[above] {$d_1^{-1}$}(m-1-4)
  (m-1-4) edge node[above] {$d_2^{-1}$}(m-1-5)
  
  (m-2-1) edge (m-2-2) 
  (m-2-2) edge node[above] {$d_0^0$}(m-2-3)
  (m-2-3) edge node[above] {$d_1^0$}(m-2-4)
  (m-2-4) edge node[above] {$d_2^0$}(m-2-5) 
 ;
 \end{tikzpicture} 
\end{center}

We may also choose lifts of $\phi^{-1}_E$ and $\phi^0_E$ to the specified injective resolutions. Such choices, of course, always exist. However, for certain applications, we will need to work with specific choices of such lifts. As such, it is useful to specify choices of lifts in advance,
\begin{center}
 \begin{tikzpicture}[description/.style={fill=white,inner sep=2pt}]
 \matrix (m) [matrix of math nodes, row sep=2em, column sep=2.5em, text height=1.5ex, text depth=0.25ex]
 {  0 & \Phi^{-1}(E^0) & \Phi^{-1}(I_0^0) & \Phi^{-1}(I_1^0) & \cdots \\
    0 & E^{-1} & I^{-1}_0 & I^{-1}_1 & \cdots \\ 
    0 & E^0 & I_0^0 & I_1^0 & \cdots. \\
    };
 \path[->,font=\scriptsize]
  (m-1-1) edge (m-1-2) 
  (m-1-2) edge node[above] {$\Phi^{-1}(d_0^0)$}(m-1-3)
  (m-1-3) edge node[above] {$\Phi^{-1}(d_1^0)$}(m-1-4)
  (m-1-4) edge node[above] {$\Phi^{-1}(d_2^0)$}(m-1-5) 

  (m-2-1) edge (m-2-2) 
  (m-2-2) edge node[above] {$d_0^{-1}$}(m-2-3)
  (m-2-3) edge node[above] {$d_1^{-1}$}(m-2-4)
  (m-2-4) edge node[above] {$d_2^{-1}$}(m-2-5)
  
  (m-3-1) edge (m-3-2) 
  (m-3-2) edge node[above] {$d_0^0$}(m-3-3)
  (m-3-3) edge node[above] {$d_1^0$}(m-3-4)
  (m-3-4) edge node[above] {$d_2^0$}(m-3-5) 
  
  (m-1-2) edge node[left] {$\phi^{-1}_E$}(m-2-2)
  (m-1-3) edge node[left] {$\phi^{-1}_0$}(m-2-3)
  (m-1-4) edge node[left] {$\phi^{-1}_1$}(m-2-4)
  
  (m-2-2) edge node[left] {$\phi^{0}_E$}(m-3-2)
  (m-2-3) edge node[left] {$\phi^{0}_0$}(m-3-3)
  (m-2-4) edge node[left] {$\phi^{0}_1$}(m-3-4)
 ;
 \end{tikzpicture} 
\end{center}

Since $\phi^{0}_E \circ \phi^{-1}_E = w$ and $\Phi(\phi^{-1}_E) \circ \phi^0_E = w$, the compositions of the lifts to the injective resolutions are homotopic to $w$. It will also be useful to specify the homotopies beforehand.
 \begin{center}
 \begin{tikzpicture}[description/.style={fill=white,inner sep=2pt}]
 \matrix (m) [matrix of math nodes, row sep=3.5em, column sep=3.5em, text height=1.5ex, text depth=0.25ex]
 {  0 & \Phi^{-1}(E^0) & \Phi^{-1}(I_0^0) & \Phi^{-1}(I_1^0) & \Phi^{-1}(I_2^0) & \cdots \\
    0 & E^{0} & I^{0}_0 & I^{0}_1 & I_2^0 & \cdots \\ };
 \path[->,font=\scriptsize]
  (m-1-1) edge (m-1-2) 
  (m-1-2) edge node[above] {$\Phi^{-1}(d_0^0)$}(m-1-3)
  (m-1-3) edge node[above] {$\Phi^{-1}(d_1^0)$}(m-1-4)
  (m-1-4) edge node[above] {$\Phi^{-1}(d_2^0)$}(m-1-5)
  (m-1-5) edge node[above] {$\Phi^{-1}(d_3^0)$}(m-1-6) 

  (m-2-1) edge (m-2-2) 
  (m-2-2) edge node[below] {$d_0^{0}$}(m-2-3)
  (m-2-3) edge node[below] {$d_1^{0}$}(m-2-4)
  (m-2-4) edge node[below] {$d_2^{0}$}(m-2-5)
  (m-2-5) edge node[below] {$d_3^0$}(m-2-6) 
  
  (m-1-2) edge node[left] {$0$}(m-2-2)
  (m-1-3) edge node[left] {$\beta_0^0$}(m-2-3)
  (m-1-4) edge node[left] {$\beta_1^0$}(m-2-4)
  (m-1-5) edge node[left] {$\beta_2^0$}(m-2-5)
  
  (m-1-4) edge node[below] {$h^0_0$}(m-2-3)
  (m-1-5) edge node[below] {$h^0_1$}(m-2-4)
  (m-1-6) edge node[below] {$h^0_2$}(m-2-5)
 ;
 \end{tikzpicture} 
 \end{center}
 where $\beta_i^0 = w_{\Phi^{-1}(I^0_j)} - \phi_j^0 \circ \phi^{-1}_j$, and
 \begin{center}
 \begin{tikzpicture}[description/.style={fill=white,inner sep=2pt}]
 \matrix (m) [matrix of math nodes, row sep=3.5em, column sep=3.5em, text height=1.5ex, text depth=0.25ex]
 {  0 & E^{-1} & I^{-1}_0 & I^{-1}_1 & I_2^{-1} &  \cdots \\
    0 & \Phi(E^{-1}) & \Phi(I_0^{-1}) & \Phi(I_1^{-1}) & \Phi(I_2^{-1}) & \cdots. \\
    };
 \path[->,font=\scriptsize]
  (m-1-1) edge (m-1-2) 
  (m-1-2) edge node[above] {$d_0^{-1}$}(m-1-3)
  (m-1-3) edge node[above] {$d_1^{-1}$}(m-1-4)
  (m-1-4) edge node[above] {$d_2^{-1}$}(m-1-5)
  (m-1-5) edge node[above] {$d_3^{-1}$}(m-1-6) 

  (m-2-1) edge (m-2-2) 
  (m-2-2) edge node[below] {$\Phi(d_0^{-1})$}(m-2-3)
  (m-2-3) edge node[below] {$\Phi(d_1^{-1})$}(m-2-4)
  (m-2-4) edge node[below] {$\Phi(d_2^{-1})$}(m-2-5)
  (m-2-5) edge node[below] {$\Phi(d_3^{-1})$}(m-2-6) 
  
  (m-1-2) edge node[left] {$0$}(m-2-2)
  (m-1-3) edge node[left] {$\beta^{-1}_0$}(m-2-3)
  (m-1-4) edge node[left] {$\beta^{-1}_1$}(m-2-4)
  (m-1-5) edge node[left] {$\beta^{-1}_2$}(m-2-5)
  
  (m-1-4) edge node[below] {$h^{-1}_0$}(m-2-3)
  (m-1-5) edge node[below] {$h^{-1}_1$}(m-2-4)
  (m-1-6) edge node[below] {$h^{-1}_2$}(m-2-5)
 ;
 \end{tikzpicture} 
 \end{center}
 where $\beta_i^{-1} = w_{I^{-1}_j} - \Phi(\phi_j^{-1}) \circ \phi^{0}_j$.

Now, we state our construction of injective resolutions.

\begin{theorem} \label{theorem: strictification injective}
 Assume that small coproducts exist in $\mathcal A$ and that $\mathcal A$ has enough injective objects. 

 Let $E$ be an object of $\op{Fact}(w)$. Choose injective resolutions of its components, lifts of $\phi^{-1}_E$ and $\phi^0_E$ to these injective resolutions, and null-homotopies of the difference of $w$ and the compositions of the lifts as above. 

 There exists a $\bigoplus$-folding, $I = (\op{tot}^{\bigoplus}(I_{\bullet})^{-1},\op{tot}^{\bigoplus}(I_{\bullet})^0,\phi^{-1}_I,\phi^0_I)$, of $I^{-1}_{\bullet}$ and $I^0_{\bullet}$ and a co-quasi-isomorphism, $d_0: E \to I$, such that
 \begin{itemize}
  \item We have equalities
  \begin{align*}
   \phi^{-1}_{2l+1,2l+1}& =  \Phi^{-l-1}(\phi^{0}_{2l+1})  : \Phi^{-l-1}(I^{-1}_{2l+1}) \to \Phi^{-l-1}(I^{0}_{2l+1}) \\
   \phi^{-1}_{2l,2l}   & =  \Phi^{-l}(\phi^{-1}_{2l}) : \Phi^{-l-1}(I^{0}_{2l}) \to \Phi^{-l}(I^{-1}_{2l}) \\
   \phi^{0}_{2l+1,2l+1} & =   \Phi^{-l}(\phi^{-1}_{2l+1}): \Phi^{-l-1}(I^{0}_{2l+1}) \to \Phi^{-l}(I^{-1}_{2l+1}) \\
   \phi^{0}_{2l,2l}   & =  \Phi^{-l}(\phi^{0}_{2l}): \Phi^{-l}(I^{-1}_{2l}) \to \Phi^{-l}(I^{0}_{2l}),
  \end{align*}
  and 
  \begin{align*}
   \phi^{-1}_{2l+1,2l}   & =  \Phi^{-l-1}(h^{-1}_{2l}) : \Phi^{-l-1}(I^{-1}_{2l+1}) \to \Phi^{-l}(I^{-1}_{2l}) \\
   \phi^{-1}_{2l,2l-1}  & =  -\Phi^{-l}(h^{0}_{2l-1}): \Phi^{-l-1}(I^{0}_{2l}) \to \Phi^{-l}(I^{0}_{2l-1}) \\
   \phi^{0}_{2l+1,2l} & =    -\Phi^{-l}(h^{0}_{2l})  : \Phi^{-l-1}(I^{0}_{2l+1}) \to \Phi^{-l}(I^{0}_{2l}) \\
   \phi^{0}_{2l,2l-1} & =   \Phi^{-l}(h^{-1}_{2l-1}) : \Phi^{-l}(I^{-1}_{2l}) \to \Phi^{-l+1}(I^{-1}_{2l-1}).
  \end{align*}
  \item $d_0$ is given by the compositions,
  \begin{align*}
   E^{-1} & \overset{d^{-1}_0}{\to} I_0^{-1} \to I^{-1} \\
   E^{0} & \overset{d^{0}_0}{\to} I_0^{0} \to I^{0}.
  \end{align*}
  \item $d_0$ is a quasi-isomorphism when both injective resolutions are finite.
 \end{itemize}
\end{theorem}

\begin{proof}
 We will construct $\phi^{-1}_{p,q}$ and $\phi^0_{p,q}$ such that
 \begin{equation} \label{equation: -1}
  \sum_{t \in \Z} \phi^0_{t,p+n} \circ \phi^{-1}_{p,t} =
 \begin{cases}
  0 & n \not = 0 \\
  w & n=0
 \end{cases}
 \end{equation}
 and
 \begin{equation} \label{equation: 0}
  \sum_{t \in \Z} \Phi(\phi^{-1}_{t,p+n}) \circ \phi^{0}_{p,t} =
 \begin{cases}
  0 & n \not = 0 \\
  w & n=0.
 \end{cases}
 \end{equation}
 We will proceed by downward induction on $n$. We begin by defining $\phi^{-1}_{p,q}$ and $\phi^0_{p,q}$ for $p-1 \leq q$ exactly as in the conclusions of the theorem. This satisfies the cases, $n=2,1,0$, of Equations~\eqref{equation: -1} and \eqref{equation: 0}. 
 
 Now assume we have constructed $\phi^{-1}_{p,q}$ and $\phi^0_{p,q}$ for $q \geq p-m$ satisfying Equations \eqref{equation: -1} and \eqref{equation: 0} for $n \geq -m+1$. We need to construct $\phi^{-1}_{s,s-m-1}$ and $\phi^{0}_{s,s-m-1}$ such that  
 \begin{equation} \label{equation: induction -1}
  \sum_{p-m \leq t \leq p} \phi^0_{t,p-m} \circ \phi^{-1}_{p,t} + \phi^0_{p-m-1,p-m} \circ \phi^{-1}_{p,p-m-1} + \phi^0_{p+1,p-m} \circ \phi^{-1}_{p,p+1} = 0
 \end{equation}
 and
 \begin{equation} \label{equation: induction 0}
  \sum_{p-m \leq t \leq p} \Phi(\phi^{-1}_{t,p-m}) \circ \phi^{0}_{p,t} + \Phi(\phi^{-1}_{p-m-1,p-m}) \circ \phi^{0}_{p,p-m-1} + \Phi(\phi^{-1}_{p+1,p-m}) \circ \phi^{0}_{p,p+1} = 0.
 \end{equation}
 We will see that solving Equation \eqref{equation: induction -1} and \eqref{equation: induction 0} amounts to choosing a null-homotopy for an acyclic chain map between complexes of injectives. Solving Equation \eqref{equation: induction -1} for $p$ even and Equation \eqref{equation: induction 0} for $p$ odd is independent from solving Equation \eqref{equation: induction -1} for $p$ odd and Equation \eqref{equation: induction 0} for $p$ even. We will solve Equation \eqref{equation: induction -1} for $p$ even and Equation \eqref{equation: induction 0} for $p$ odd. The other case is completely analogous. 
 
 Assume that $m=2r$. The case of odd $m$ follows analogously. Consider the chain complexes of injectives, $(\Phi^{-u-1}(I^0_v),\Phi^{-u-1}(d^0_v))$ and $(\Phi^{-u-r}(I^0_{v-m}),\Phi^{-u-r}(d^0_{v-m}))$. Each complex contains homology in a single degree, $0$ for $(\Phi^{-u-1}(I^0_v),\Phi^{-u-1}(d^0_v))$ and $m$ for $(\Phi^{-u-r}(I^0_{v-m}),\Phi^{-u-r}(d^0_{v-m}))$. 
 
 There are morphisms,
 \begin{align*}
  \psi_{u,2q}:= \Phi^{q-u}(\sum_{2q-m < t \leq 2q} \phi^0_{t,2q-m+1} \circ \phi^{-1}_{2q,t}) & : \Phi^{-u-1}(I^0_{2q}) \to \Phi^{-u-r}(I^0_{2q-m}) \\
  \psi_{u,2q+1}:= \Phi^{q-u}(\sum_{2q+1-m < t \leq 2q+1} \Phi(\phi^{-1}_{t,2q-m+1}) \circ \phi^{0}_{2q+1,t}) & : \Phi^{-u-1}(I^0_{2q+1}) \to \Phi^{-u-r}(I^0_{2q-m+1}).
 \end{align*}
 We claim that $\psi_u: (\Phi^{-u-1}(I^0_v),\Phi^{-u-1}(d^0_v)) \to (\Phi^{-u-r}(I^0_{v-m}),\Phi^{-u-r}(d^0_{v-m}))$ is a chain map. Let us assume the validity of this claim for the moment and continue. Since $\psi_u$ must induce the trivial map on the homology of the complexes and the components of the complexes are injectives, there exists a null-homotopy,
 \begin{displaymath}
  h_{u,v}: \Phi^{-u-1}(I^0_v) \to \Phi^{-u-r}(I^0_{v-m-1}),
 \end{displaymath}
 of $\psi_u$. Let us draw the diagram for the homotopy. Recall that $\phi^{-1}_{2q,2q+1} = -\Phi^{-q-1}(d_{2q+1}^0)$ and $\phi^{0}_{2q-1,2q} = -\Phi^{-q}(d_{2q}^0)$.
 \begin{center}
 \begin{tikzpicture}[description/.style={fill=white,inner sep=2pt}]
 \matrix (m) [matrix of math nodes, row sep=4.5em, column sep=4.5em, text height=1ex, text depth=0.20ex]
 {  \Phi^{-u-1}(I^0_{2q-1}) & \Phi^{-u-1}(I^0_{2q}) & \Phi^{-u-1}(I^0_{2q+1}) & \Phi^{-u-1}(I^0_{2q+2}) \\
    \Phi^{-u-r}(I^0_{2q-m-1}) & \Phi^{-u-r}(I^0_{2q-m}) & \Phi^{-u-r}(I^0_{2q+m+1}) & \Phi^{-u-r}(I^0_{2q+m+2}) \\
    };
 \path[->,font=\scriptsize]
  (m-1-1) edge (m-1-2) 
  (m-1-2) edge node[above] {$-\Phi^{q-u}(\phi^{-1}_{2q,2q+1})$}(m-1-3)
  (m-1-3) edge node[above] {$-\Phi^{q-u}(\phi^0_{2q+1,2q+2})$}(m-1-4)

  (m-2-1) edge node[below] {$-\Phi^{q-u}(\phi^0_{2q-m-1,2q-m})$}(m-2-2) 
  (m-2-2) edge node[below] {$-\Phi^{q-u+1}(\phi^{-1}_{2q-m,2q-m+1})$}(m-2-3)
  (m-2-3) edge (m-2-4)
  
  (m-1-2) edge node[left] {$\psi_{u,2q}$}(m-2-2)
  (m-1-3) edge node[left] {$\psi_{u,2q+1}$}(m-2-3)
  
  (m-1-2) edge node[above] {$h_{u,2q}$}(m-2-1)
  (m-1-3) edge node[above] {$h_{u,2q+1}$}(m-2-2)
  (m-1-4) edge node[above] {$h_{u,2q+2}$}(m-2-3)
 ;
 \end{tikzpicture} 
 \end{center}

 We can rewrite the equations for the homotopy,
 \begin{align*}
  \psi_{u,2q} & = - h_{u,2q+1} \circ \Phi^{q-u}(\phi^{-1}_{2q,2q+1}) - \Phi^{q-u}(\phi^0_{2q-m-1,2q-m}) \circ h_{u,2q} \\
  \psi_{u,2q+1} & = - h_{u,2q+2} \circ \Phi^{q-u}(\phi^0_{2q+1,2q+2}) -\Phi^{q-u+1}(\phi^{-1}_{2q-m,2q-m+1}) \circ h_{u,2q+1},
 \end{align*}
 as 
 \begin{gather*}
  \sum_{2q-m < t \leq 2q} \phi^0_{t,2q-m+1} \circ \phi^{-1}_{2q,t} + \Phi^{u-q}(h_{q,2q+1}) \circ \phi^{-1}_{2q,2q+1} + \phi^0_{2q-m-1,2q-m} \circ \Phi^{u-q}(h_{u,2q}) = 0  \\
  \sum_{2q-m+1 < t \leq 2q+1} \Phi(\phi^{-1}_{t,2q-m+1}) \circ \phi^{0}_{2q+1,t} + \Phi(\phi^{-1}_{2q-m,2q-m+1}) \circ \Phi^{u-q}(h_{u,2q+1}) \\ + \Phi^{u-q+1}(h_{u,2q+2}) \circ \phi^{0}_{2q+1,2q+2} = 0.
 \end{gather*}
 We then set
 \begin{align*}
  \phi^0_{2q+1,2q-m} &:= \Phi^{u-q}(h_{u,2q+1}) \\
  \phi^{-1}_{2q,2q-m-1} &:= \Phi^{u-q}(h_{u,2q})
 \end{align*}
 to solve Equation \eqref{equation: induction -1} for $p=2q$ and Equation \eqref{equation: induction 0} for $p=2q+1$.
 
 Thus, we have constructed $\phi^{-1}_{s,s-m-1}$ and $\phi^0_{s,s-m-1}$ completing the induction step if we can show that $\psi_u$ is a chain map. We check the commutativity of the square,
 \begin{center}
 \begin{tikzpicture}[description/.style={fill=white,inner sep=2pt}]
 \matrix (m) [matrix of math nodes, row sep=3.5em, column sep=6.5em, text height=1.5ex, text depth=0.25ex]
 {  \Phi^{-u-1}(I^0_{2q}) & \Phi^{-u-1}(I^0_{2q+1}) \\
    \Phi^{-u-r}(I^0_{2q-m}) & \Phi^{-u-r}(I^0_{2q+m+1}), \\
    };
 \path[->,font=\scriptsize]
  (m-1-1) edge node[above] {$-\Phi^{q-u}(\phi^{-1}_{2q,2q+1})$}(m-1-2)
 
  (m-2-1) edge node[below] {$-\Phi^{q-u+1}(\phi^{-1}_{2q-m,2q-m+1})$}(m-2-2)
  
  (m-1-1) edge node[left] {$\psi_{u,2q}$}(m-2-1)
  (m-1-2) edge node[left] {$\psi_{u,2q+1}$}(m-2-2)
 ;
 \end{tikzpicture} 
 \end{center} 
 as the other squares are handled similarly. Commutativity of the above square is equivalent to the equality,
 \begin{gather} \label{equation: chain map}
  \left(\sum_{2q-m+1 \leq t \leq 2q+1} \Phi(\phi^{-1}_{t,2q-m+1}) \circ \phi^0_{2q+1,t}\right) \circ \phi^{-1}_{2q,2q+1} = \Phi(\phi^{-1}_{2q-m,2q+m-1}) \circ \left(\sum_{2q-m \leq t \leq 2q} \phi^0_{t,2q-m} \circ \phi^{-1}_{2q,t} \right).
 \end{gather}
 From the induction hypothesis, for $2q-m+1 \leq t \leq 2q+1$, we have 
 \begin{displaymath}
  \phi^0_{2q+1,t} \circ \phi^{-1}_{2q,2q+1} = - \sum_{t-1 \leq s \leq 2q} \phi^0_{s,t} \circ \phi^{-1}_{2q,s},
 \end{displaymath}
 and, for $2q-m \leq t \leq 2q$,
 \begin{displaymath}
  \Phi(\phi^{-1}_{2q-m,2q+m-1}) \circ \phi^0_{t,2q-m} = - \sum_{2q-m-1 \leq s \leq t+1} \Phi(\phi^{-1}_{s,2q+m-1}) \circ \phi^0_{t,s}.
 \end{displaymath}
 Thus, both sides of Equation \eqref{equation: chain map} are equal to
 \begin{displaymath}
  -\sum_{2q-m-1 \leq s \leq t+1 \leq 2q} \Phi(\phi^{-1}_{s,2q+m-1}) \circ \phi^0_{t,s} \circ \phi^{-1}_{2q,t}.
 \end{displaymath}
 This finishes the construction of the factorization, $(I^{-1},I^0,\phi^{-1}_I,\phi^0_I)$. 
 
 We turn to checking that $d_0$, as defined in the conclusion of the theorem, is a morphism of factorizations. By the construction of $\phi^{-1}_I$ and $\phi^0_I$, to check that the bolded squares in 
 \begin{center}
 \begin{tikzpicture}[description/.style={fill=white,inner sep=2pt}]
 \matrix (m) [matrix of math nodes, row sep=3.5em, column sep=6em, text height=1.5ex, text depth=0.25ex]
 {  \Phi^{-1}(E^0) & E^{-1} & E^{0} \\
    \Phi^{-1}(I^0_0) & I^{-1}_0 \oplus \Phi^{-1}(I^{0}_1) & I^{0}_0 \oplus I^{-1}_1 \oplus \Phi^{-1}(I^0_2) \\
    \Phi^{-1}(I^0_0) & I^{-1}_0 & I^{0}_0 \\
    };
 \path[->,font=\scriptsize]
  (m-1-1) edge[thick] node[above] {$\phi^{-1}_E$}(m-1-2)
  (m-1-2) edge[thick] node[above] {$\phi^{0}_E$}(m-1-3)
  (m-2-1) edge node[above] {$\begin{pmatrix} \phi^{-1}_0 & \Phi^{-1}(d_0^{-1}) \end{pmatrix}$}(m-2-2)
  (m-2-2) edge node[above] {$\begin{pmatrix} \phi^{0}_0 & h_0^{0} \\ d^{-1}_1 & \phi^{-1}_1 \\ 0 & \Phi^{-1}(d_2^0)  \end{pmatrix}$}(m-2-3)
  (m-3-1) edge[thick] node[above] {$\phi^{-1}_I$}(m-3-2)
  (m-3-2) edge[thick] node[above] {$\phi^{0}_I$}(m-3-3)
  
  (m-1-1) edge[thick] node[left] {$\Phi^{-1}(d_0^0)$}(m-2-1)
  (m-1-2) edge[thick] node[left] {$\begin{pmatrix} d^{-1}_0 & 0 \end{pmatrix}$}(m-2-2)
  (m-1-3) edge[thick] node[right] {$\begin{pmatrix} d^{0}_0 & 0 & 0 \end{pmatrix}$}(m-2-3)
  (m-2-1) edge[thick] (m-3-1)
  (m-2-2) edge[thick] (m-3-2)
  (m-2-3) edge[thick] (m-3-3)
 ;
 \end{tikzpicture} 
 \end{center} 
 commute, it suffices to show that the upper squares commute. This is immediate.

 Finally, we demonstrate that the cone of $d_0$ is co-acyclic. The factorization $E$ folds the trivial complexes $(E^{-1},E^0)$ while $I$ folds $(I^{-1}_{\bullet},I^0_{\bullet})$. The cone over $d_0$ folds the cones of the morphisms of the chain complexes $d_0^{-1}: E^{-1} \to I^{-1}_{\bullet}$ and $d_0^0: E^0 \to I^0_{\bullet}$. Each of these complexes is bounded below and acyclic thus we may apply Lemma~\ref{lemma: acyclic folds to acyclic} to conclude that the cone of $d_0$ is co-acylic and the cone of $d_0$ is acyclic if the chosen injective resolutions are bounded.
\end{proof}

There is a special situation where the components, $\phi^{-1}_{p,q}$ and $\phi^0_{p,q}$, vanish for $q < p-1$. 

\begin{corollary} \label{corollary: easier construction}
 Assume that
 \begin{align*}
  h^{-1}_p \circ \phi^{-1}_{p+1} = \Phi(\phi^{-1}_p) \circ h^0_{p} \\
  \Phi(h^{0}_p) \circ \phi^{0}_{p+1} = \Phi(\phi^{0}_p) \circ h^{-1}_{p} \\
  \Phi(h^{-1}_{p-1}) \circ h^{-1}_p = 0 \\
  \Phi(h^{0}_{p-1}) \circ h^{0}_p = 0.
 \end{align*}
 Then, in the factorization constructed in Theorem \ref{theorem: strictification injective}, we may take 
 \begin{displaymath}
  \phi^{-1}_{p,q} = \phi^0_{p,q} = 0
 \end{displaymath}
 for $q < p-1$.
\end{corollary}

\begin{proof}
 Under the hypotheses, we can take $\phi^{-1}_{p,q} = \phi^0_{p,q} = 0$ for $q < p-1$ and satisfy Equations \eqref{equation: -1} and \eqref{equation: 0} for all $n$.
\end{proof}

We also have the dual statement which we record in full detail for ease of future reference. Assume that $\mathcal A$ has enough projectives. Let $E$ be an object of $\op{Fact}(w)$. Choose projective resolutions of its components and lifts of $\phi^{-1}_E$ and $\phi^0_E$ to the those resolutions,
\begin{center}
 \begin{tikzpicture}[description/.style={fill=white,inner sep=2pt}]
 \matrix (m) [matrix of math nodes, row sep=2em, column sep=2.5em, text height=1.5ex, text depth=0.25ex]
 {  \cdots & \Phi^{-1}(P^0_{-1}) & \Phi^{-1}(P_0^0) & \Phi^{-1}(E^0) & 0 \\
    \cdots & P^{-1}_{-1} & P^{-1}_0 & E^{-1} & 0 \\ 
    \cdots & P^0_{-1} & P_0^0 & E^0 & 0. \\
    };
 \path[->,font=\scriptsize]
  (m-1-4) edge (m-1-5) 
  (m-1-3) edge node[above] {$\Phi^{-1}(d_0^0)$}(m-1-4)
  (m-1-2) edge node[above] {$\Phi^{-1}(d_{-1}^0)$}(m-1-3)
  (m-1-1) edge node[above] {$\Phi^{-1}(d_{-2}^0)$}(m-1-2) 

  (m-2-4) edge (m-2-5) 
  (m-2-3) edge node[above] {$d_0^{-1}$}(m-2-4)
  (m-2-2) edge node[above] {$d_{-1}^{-1}$}(m-2-3)
  (m-2-1) edge node[above] {$d_{-2}^{-1}$}(m-2-2)
  
  (m-3-4) edge (m-3-5) 
  (m-3-3) edge node[above] {$d_0^0$}(m-3-4)
  (m-3-2) edge node[above] {$d_{-1}^0$}(m-3-3)
  (m-3-1) edge node[above] {$d_{-2}^0$}(m-3-2) 
  
  (m-1-2) edge node[left] {$\phi^{-1}_{-1}$}(m-2-2)
  (m-1-3) edge node[left] {$\phi^{-1}_0$}(m-2-3)
  (m-1-4) edge node[left] {$\phi^{-1}_E$}(m-2-4)
  
  (m-2-2) edge node[left] {$\phi^{0}_{-1}$}(m-3-2)
  (m-2-3) edge node[left] {$\phi^{0}_0$}(m-3-3)
  (m-2-4) edge node[left] {$\phi^{0}_E$}(m-3-4)
 ;
 \end{tikzpicture} 
\end{center}
Also choose null-homotopies,
 \begin{center}
 \begin{tikzpicture}[description/.style={fill=white,inner sep=2pt}]
 \matrix (m) [matrix of math nodes, row sep=3.5em, column sep=3.5em, text height=1.5ex, text depth=0.25ex]
 {  \cdots & \Phi^{-1}(P^0_{-2}) & \Phi^{-1}(P_{-1}^0) & \Phi^{-1}(P_0^0) & \Phi^{-1}(E^0) & 0 \\
    \cdots & P^{0}_{-2} & P^{0}_{-1} & P^{0}_0 & E^0 & 0 \\ };
 \path[->,font=\scriptsize]
  (m-1-5) edge (m-1-6) 
  (m-1-4) edge node[above] {$\Phi^{-1}(d_0^0)$}(m-1-5)
  (m-1-3) edge node[above] {$\Phi^{-1}(d_{-2}^0)$}(m-1-4)
  (m-1-2) edge node[above] {$\Phi^{-1}(d_{-2}^0)$}(m-1-3)
  (m-1-1) edge node[above] {$\Phi^{-1}(d_{-3}^0)$}(m-1-2) 

  (m-2-5) edge (m-2-6) 
  (m-2-4) edge node[below] {$d_0^{0}$}(m-2-5)
  (m-2-3) edge node[below] {$d_{-1}^{0}$}(m-2-4)
  (m-2-2) edge node[below] {$d_{-2}^{0}$}(m-2-3)
  (m-2-1) edge node[below] {$d_{-3}^0$}(m-2-2) 
  
  (m-1-5) edge node[left] {$0$}(m-2-5)
  (m-1-4) edge node[left] {$\beta^0_0$}(m-2-4)
  (m-1-3) edge node[left] {$\beta^0_{-1}$}(m-2-3)
  (m-1-2) edge node[left] {$\beta^0_{-2}$}(m-2-2)
  
  (m-1-4) edge node[below] {$h^0_{-1}$}(m-2-3)
  (m-1-3) edge node[below] {$h^0_{-2}$}(m-2-2)
  (m-1-2) edge node[below] {$h^0_{-3}$}(m-2-1)
 ;
 \end{tikzpicture} 
 \end{center}
 where $\beta^0_j = w_{\Phi^{-1}(P^0_j)} - \phi^0_j \circ \phi^{-1}_j$, and
 \begin{center}
 \begin{tikzpicture}[description/.style={fill=white,inner sep=2pt}]
 \matrix (m) [matrix of math nodes, row sep=3.5em, column sep=3.5em, text height=1.5ex, text depth=0.25ex]
 {  \cdots & P^{-1}_{-2} & P_{-1}^{-1} & P_0^{-1} & E^{-1} & 0 \\
    \cdots & \Phi(P^{-1}_{-2}) & \Phi(P^{-1}_{-1}) & \Phi(P^{-1}_0) & \Phi(E^{-1}) & 0 \\ };
 \path[->,font=\scriptsize]
  (m-1-5) edge (m-1-6) 
  (m-1-4) edge node[above] {$d_0^{-1}$}(m-1-5)
  (m-1-3) edge node[above] {$d_{-2}^{-1}$}(m-1-4)
  (m-1-2) edge node[above] {$d_{-2}^{-1}$}(m-1-3)
  (m-1-1) edge node[above] {$d_{-3}^{-1}$}(m-1-2) 

  (m-2-5) edge (m-2-6) 
  (m-2-4) edge node[below] {$\Phi(d_0^{-1})$}(m-2-5)
  (m-2-3) edge node[below] {$\Phi(d_{-1}^{-1})$}(m-2-4)
  (m-2-2) edge node[below] {$\Phi(d_{-2}^{-1})$}(m-2-3)
  (m-2-1) edge node[below] {$\Phi(d_{-3}^{-1})$}(m-2-2) 
  
  (m-1-5) edge node[left] {$0$}(m-2-5)
  (m-1-4) edge node[left] {$\beta^{-1}_0$}(m-2-4)
  (m-1-3) edge node[left] {$\beta^{-1}_{-1}$}(m-2-3)
  (m-1-2) edge node[left] {$\beta^{-1}_{-2}$}(m-2-2)
  
  (m-1-4) edge node[below] {$h^{-1}_{-1}$}(m-2-3)
  (m-1-3) edge node[below] {$h^{-1}_{-2}$}(m-2-2)
  (m-1-2) edge node[below] {$h^{-1}_{-3}$}(m-2-1)
 ;
 \end{tikzpicture} 
 \end{center}
 where $\beta^{-1}_j = w_{P^{-1}_j} - \Phi(\phi^{-1}_j) \circ \phi^{0}_j$.

\begin{theorem} \label{theorem: strictification projective}
 Assume that small products exist in $\mathcal A$ and that $\mathcal A$ has enough projective objects. 

 Let $E$ be an object of $\op{Fact}(w)$. Choose projective resolutions of its components, lifts of $\phi^{-1}_E$ and $\phi^0_E$ to these projective resolutions, and null-homotopies of the difference of $w$ and the compositions of the lifts as above. 

 There exists a factorization, $P = (\op{tot}^{\prod}(P_{\bullet})^{-1},\op{tot}^{\prod}(P_{\bullet})^0,\phi^{-1}_P,\phi^0_P)$, $\prod$-folding $P_{\bullet}^{-1}$ and $P_{\bullet}^0$ and a contra-quasi-isomorphism, $d_0: P \to E$, such that
 \begin{itemize}
  \item The components of $\phi^{-1}_{p,q}$ and $\phi^0_{p,q}$ with $q = p, p-1$ are given by
  \begin{align*}
   \Phi^{-l-1}(\phi^{0}_{2l+1}) & = \phi^{-1}_{2l+1,2l+1} : \Phi^{-l-1}(P^{-1}_{2l+1}) \to \Phi^{-l-1}(P^{0}_{2l+1}) \\
   \Phi^{-l}(\phi^{-1}_{2l}) & = \phi^{-1}_{2l,2l} : \Phi^{-l-1}(P^{0}_{2l}) \to \Phi^{-l}(P^{-1}_{2l}) \\
   \Phi^{-l}(\phi^{-1}_{2l+1}) & = \phi^{0}_{2l+1,2l+1} : \Phi^{-l-1}(P^{0}_{2l+1}) \to \Phi^{-l}(P^{-1}_{2l+1}) \\
   \Phi^{-l}(\phi^{0}_{2l}) & = \phi^{0}_{2l,2l} : \Phi^{-l}(P^{-1}_{2l}) \to \Phi^{-l}(P^{0}_{2l}).
  \end{align*}
  and
  \begin{align*}
   \Phi^{-l-1}(h^{-1}_{2l}) & = \phi^{-1}_{2l+1,2l} : \Phi^{-l-1}(P^{-1}_{2l+1}) \to \Phi^{-l}(P^{-1}_{2l}) \\
   -\Phi^{-l}(h^{0}_{2l-1}) & = \phi^{-1}_{2l,2l-1} : \Phi^{-l-1}(P^{0}_{2l}) \to \Phi^{-l}(P^{0}_{2l-1}) \\
   -\Phi^{-l}(h^{0}_{2l}) & = \phi^{0}_{2l+1,2l} : \Phi^{-l-1}(P^{0}_{2l+1}) \to \Phi^{-l}(P^{0}_{2l}) \\
   \Phi^{-l}(h^{-1}_{2l-1}) & = \phi^{0}_{2l,2l-1} : \Phi^{-l}(P^{-1}_{2l}) \to \Phi^{-l+1}(P^{-1}_{2l-1}).
  \end{align*}
  \item $d_0$ is given by the compositions,
  \begin{align*}
   P^{-1} & \overset{d^{-1}_0}{\to} P_0^{-1} \to E^{-1} \\
   P^{0} & \overset{d^{0}_0}{\to} P_0^{0} \to E^{0}.
  \end{align*}
  \item $d_0$ is a quasi-isomorphism when both injective resolutions are finite.
 \end{itemize}
 
 Furthermore, if 
 \begin{align*}
  h^{-1}_p \circ \phi^{-1}_{p+1} & = \Phi(\phi^{-1}_p) \circ h^0_{p} \\
  \Phi(h^{0}_p) \circ \phi^{0}_{p+1} & = \Phi(\phi^{0}_p) \circ h^{-1}_{p} \\
  \Phi(h^{-1}_{p-1}) \circ h^{-1}_p & = 0 \\
  \Phi(h^{0}_{p-1}) \circ h^{0}_p & = 0.
 \end{align*}
 Then, we may take 
 \begin{displaymath}
  \phi^{-1}_{p,q} = \phi^0_{p,q} = 0
 \end{displaymath}
 for $q < p-1$.
\end{theorem}

\begin{proof}
 The statement is dual to those for Theorem \ref{theorem: strictification injective} and Corollary \ref{corollary: easier construction}.  Therefore, we may replace $\mathcal A$ by its opposite category.
\end{proof}

\begin{remark}
 The classical case of this construction is to let $R$ be a commutative Noetherian regular $k$-algebra, $\mathcal A = \op{mod}R$, $\Phi = \op{Id}$, and $w \in R$. We then consider an ideal $I$ containing $w$ and  generated by a regular sequence $(x_1, \ldots, x_n)$ so that we may write $w = \sum w_i x_i$.  We may consider the factorization $(0, R/I, 0 , 0)$.  The Koszul complex on $(x_1, \ldots, x_n)$ gives a projective resolution of $R/I$ and contraction with $(w_1, \ldots, w_n)$ gives a homotopy $h$ such that $h^2 =0$. The projective replacement 
 \[
 (\bigoplus_{l=0}^{\lfloor \frac{n}{2} \rfloor} \bigwedge\nolimits^{2l}\langle x_1, \ldots, x_n \rangle \otimes_k R, \bigoplus_{l=0}^{\lfloor \frac{n}{2} \rfloor} \bigwedge\nolimits^{2l+1}\langle x_1, \ldots, x_n \rangle \otimes_k R, d+h, d+h)
 \]
 of $(0, R/I, 0 , 0)$ is called the stabilization of $R/I$. This recovers Eisenbud's original construction \cite{EisMF}.
\end{remark}

As a first application, we give a spectral sequence for computing morphisms in $\op{D}^{\op{abs}}(\op{Fact} w)$. 

\begin{lemma} \label{lemma: spectral sequence injective}
 Let $E$ and $F$ be two factorizations of $w$. Assume that $\mathcal A$ has enough injectives and small coproducts, and assume that coproducts of injectives are injective.
 
 There is a spectral sequence whose $E_1$-page is 
 \begin{displaymath}
  E^{p,q}_1 = 
  \begin{cases} 
   \op{Ext}^{p+q-1}_{\mathcal A}(E^{-1},\Phi^{-s}(F^0)) \oplus \op{Ext}^{p+q}_{\mathcal A}(E^0,\Phi^{-s}(F^0)) & p = 2s \\ 
   \op{Ext}^{p+q-1}_{\mathcal A}(E^{-1},\Phi^{-s}(F^{-1})) \oplus \op{Ext}^{p+q}_{\mathcal A}(E^0,\Phi^{-s-1}(F^{-1})) & p = 2s+1.
  \end{cases}
 \end{displaymath}
 If the components of $F$ have finite injective dimension, the spectral sequence strongly converges to $\bigoplus_r \op{Hom}(E,F[r])$ taken in $\op{D}^{\op{co}}(\op{Fact} w)$ or $\op{D}^{\op{abs}}(\op{Fact} w)$.
\end{lemma}

\begin{proof}
 Choose finite injective resolutions of $F^{-1}$ and $F^0$,
 \begin{align*}
  0 & \to F^{-1} \to I^{-1}_0 \to I^{-1}_1 \to \cdots \\
  0 & \to F^0 \to I^0_0 \to I^0_1 \to \cdots,
 \end{align*}
 and use Theorem \ref{theorem: strictification injective} to construct a co-quasi-isomorphic resolution, $I$, of $F$.
 
 Filter the complex, $\textbf{Hom}_w^*(E,I)$, by
 \begin{gather*}
  \mathcal F^p \textbf{Hom}_w^n(E,I) := \{ (g^{-1},g^0) \mid \\
  \begin{cases}
   g^{-1}(E^{-1}) \subseteq \bigoplus_{2l \leq n+p-1} \Phi^{m-l}(I^{-1}_{2l}) \oplus \bigoplus_{2l+1 \leq n+p-1} \Phi^{m-l-1}(I^0_{2l+1}) \\
   g^{0}(E^{0}) \subseteq \bigoplus_{2l \leq n+p} \Phi^{m-l}(I^{0}_{2l}) \oplus \bigoplus_{2l+1 \leq n+p} \Phi^{m-l}(I^{-1}_{2l+1}) & n=2m \\
   g^{-1}(E^{-1}) \subseteq \bigoplus_{2l \leq n+p-1} \Phi^{m-l}(I^{0}_{2l}) \oplus \bigoplus_{2l+1 \leq n+p-1} \Phi^{m-l}(I^{-1}_{2l+1}) \\
   g^{0}(E^{0}) \subseteq \bigoplus_{2l \leq n+p} \Phi^{m-l+1}(I^{-1}_{2l}) \oplus \bigoplus_{2l+1 \leq n+p} \Phi^{m-l}(I^{-1}_{2l+1}) & n=2m+1 \}.
  \end{cases}
 \end{gather*}
 The associated graded complex is
 \begin{gather*}
  \op{Gr}^p \textbf{Hom}_w^n(E,I) := 
  \begin{cases}
   \op{Hom}_{\mathcal A}(E^{-1}, \Phi^{-q}(I_{p+n-1}^{0})) \oplus \op{Hom}_{\mathcal A}(E^{0}, \Phi^{-q}(I_{p+n}^{0})) & p=2q \\
   \op{Hom}_{\mathcal A}(E^{-1}, \Phi^{-q}(I_{p+n-1}^{-1})) \oplus \op{Hom}_{\mathcal A}(E^{0}, \Phi^{-q-1}(I_{p+n}^{-1})) & p=2q+1.
  \end{cases}
 \end{gather*}
 with differentials given by composition with the differentials in the complexes $I^{-1}_{\bullet}$ and $I^{0}_{\bullet}$. We set
 \begin{displaymath}
  E^{p,q}_0 := \op{Gr}^p \textbf{Hom}_w^q(E,I)
 \end{displaymath}
 to start our spectral sequence. The $E_1$-page is as above. 
 
 If we assume that the components of $F$ have injective resolutions of length $t$, then the spectral sequence degenerates at the $(t+1)$-st page.
\end{proof}

\begin{lemma} \label{lemma: spectral sequence projective}
 Let $E$ and $F$ be two factorizations of $w$. Assume that $\mathcal A$ has enough projectives and small products. Further, assume that products of projectives remain projective.
 
 There is a spectral sequence whose $E_1$-page is 
 \begin{displaymath}
  E^{p,q}_1 = 
  \begin{cases} 
   \op{Ext}^{p+q-1}_{\mathcal A}(E^{-1},\Phi^{-s}(F^0)) \oplus \op{Ext}^{p+q}_{\mathcal A}(E^0,\Phi^{-s}(F^0)) & p = 2s \\ 
   \op{Ext}^{p+q-1}_{\mathcal A}(E^{-1},\Phi^{-s}(F^{-1})) \oplus \op{Ext}^{p+q}_{\mathcal A}(E^0,\Phi^{-s-1}(F^{-1})) & p = 2s+1.
  \end{cases}
 \end{displaymath}
 If the components of $E$ have finite projective dimension, the spectral sequence strongly converges to $\bigoplus_r \op{Hom}_{\op{D}^{\op{ctr}}(\op{Fact} w)}(E,F[r])$.
\end{lemma}

\begin{proof}
 We may replace $\mathcal A$ by its opposite category and apply Lemma \ref{lemma: spectral sequence injective}.
\end{proof}

\section{Derived factored functors and some applications of the resolutions} \label{section: applications}

\begin{definition}
Consider two triples, $(\mathcal A, \Phi, w)$ and $(\mathcal B, \Psi, v)$. An additive functor, $\theta: \mathcal A \to \op{Ch}^{\op{b}}(\mathcal B)$, is called \newterm{factored} if there is a natural isomorphism
\begin{displaymath}
\epsilon_{\theta}:  \theta \circ \Phi \cong \Psi \circ \theta
\end{displaymath}
and 
\begin{displaymath}
 \theta(w_A) = v_{\theta(A)} : \theta(A) \to \theta(\Phi(A)) \cong \Psi(\theta(A)).
\end{displaymath}
for all objects, $A \in \mathcal A$. Here we extend $\Psi$ and $v$, in the obvious manner, to the Abelian category $\op{Ch}^{\op{b}}(\mathcal B)$.
\end{definition}

A factored functor induces a functor on factorization categories,
 \begin{align*}
  \theta^f : \op{Fact}(w) & \to \op{Fact}(v) \\
  E & \mapsto \op{tot}(\theta(E^{-1}),\theta(E^{0}),\theta(\phi^{-1}_E),\theta(\phi^0_E))
 \end{align*}
 where $\op{tot}$ is the totalization of the complex of factorizations $(\theta(E^{-1}),\theta(E^{0}),\theta(\phi^{-1}_E),\theta(\phi^0_E))$.

\begin{definition}
 Let $\theta$ be a left exact factored functor. A factorization, $E$ is called $\theta$-\newterm{adapted} if the components of $E$ are $\theta$-adapted in the usual sense. Recall that this means that whenever we have an exact sequence 
 \begin{displaymath}
  0 \to E^i \to A \to A^{\prime} \to 0
 \end{displaymath}
 is a short exact sequence in $\mathcal A$ for $i=-1,0$, then 
 \begin{displaymath}
  0 \to \theta(E^i) \to \theta(A) \to \theta(A^{\prime}) \to 0 
 \end{displaymath}
 is a short exact sequence in $\op{Ch}(\mathcal B)$. 
 
 Let $\theta$ be a right exact factored functor. A factorization, $E$ is called $\theta$-\newterm{adapted} if the components of $E$ are $\theta$-adapted in the usual sense. 
%Recall that this means that whenever we have an exact sequence 
% \begin{displaymath}
%  0 \to A \to A^{\prime} \to E^i \to 0
% \end{displaymath}
% is a short exact sequence in $\mathcal A$ for $i=-1,0$, then 
% \begin{displaymath}
%  0 \to \theta(A) \to \theta(A^{\prime}) \to \theta(E^i) \to 0 
% \end{displaymath}
% is a short exact sequence in $\op{Ch}(\mathcal B)$. 
% 
 In either of these cases, let $\mathcal Adp_{\theta}$ be the full additive subcategory of $\mathcal A$ consisting of $\theta$-adapted objects.
\end{definition}
 
\begin{definition}
 If one of the following assumptions is satisfied
 \begin{itemize}
  \item $\mathcal A$ has small coproducts, $\theta$ is left exact and commutes with coproducts, and the full subcategory of $\theta$-adapted objects in $\mathcal A$ satisfies the conditions of Proposition~\ref{proposition: adapted co-replacement}, or
  \item $\mathcal A$ has small products, $\theta$ is left exact and commutes with products, and the full subcategory of $\theta$-adapted objects in $\mathcal A$ satisfies the conditions of Proposition~\ref{proposition: adapted contra-replacement}, or
  \item $\theta$ is left or right exact and the full subcategory of $\theta$-adapted objects in $\mathcal A$ satisfies the conditions of Proposition~\ref{proposition: adapted abs-replacement}, 
 \end{itemize}
 then we say there is \newterm{enough $\theta$-adapted factorizations}. 
 
 If $\mathcal A$ has small coproducts, $\theta$ is left exact and commutes with coproducts, and there is enough $\theta$-adapted factorizations, we define the \newterm{right co-derived factored functor} of $\theta$
 to be the composition 
 \begin{displaymath}
  \mathbf{R}\theta^f : \dcoFact{w} \overset{\op{Q}^{-1}_{\mathcal Adp_{\theta}}}{\to} \dcoFact{\mathcal Adp_{\theta},w} \overset{\theta^f}{\to} \dcoFact{v},
 \end{displaymath}
 where $\op{Q}_{\mathcal Adp_{\theta}}$ is the equivalence of Proposition~\ref{proposition: adapted co-replacement}.
 
 If $\mathcal A$ has small products, $\theta$ is right exact and commutes with coroducts, and there is enough $\theta$-adapted factorizations, we define the \newterm{left contra-derived factored functor} of $\theta$
 to be the composition 
 \begin{displaymath}
  \mathbf{L}\theta^f : \dconFact{w} \overset{\op{Q}^{-1}_{\mathcal Adp_{\theta}}}{\to} \dconFact{\mathcal Adp_{\theta},w} \overset{\theta^f}{\to} \dconFact{v},
 \end{displaymath}
 where $\op{Q}_{\mathcal Adp_{\theta}}$ is the equivalence of Proposition~\ref{proposition: adapted contra-replacement}.
 
 If $\theta$ is left exact/right exact and there is enough $\theta$-adapted factorizations, we define the \newterm{right/left absolutely-derived factored functor} of $\theta$
 to be the composition 
 \begin{displaymath}
  \mathbf{R}\theta^f/\mathbf{L}\theta^f : \dabsFact{w} \overset{\op{Q}^{-1}_{\mathcal Adp_{\theta}}}{\to} \dabsFact{\mathcal Adp_{\theta},w} \overset{\theta^f}{\to} \dabsFact{v},
 \end{displaymath}
 where $\op{Q}_{\mathcal Adp_{\theta}}$ is the equivalence of Proposition~\ref{proposition: adapted abs-replacement}.
\end{definition}

\begin{remark}
 Injective objects are always adapted. Therefore, for right derived functors to exist, it is simplest to assume that $\mathcal A$ has small coproducts and enough injectives, and that coproducts of injectives are injective.  This happens, e.g., for quasi-coherent sheaves over a Noetherian scheme.
\end{remark}

\begin{lemma}
Given two factored functors,
\[
\theta :  \mathcal A \to  \op{Ch}^{\op{b}} (\mathcal B)
\]
and
\[
\gamma : \mathcal B \to  \op{Ch}^{\op{b}} (\mathcal C),
\]
the functor 
\begin{align*}
\op{tot} (\gamma \circ \theta): \mathcal A & \to \op{Ch}^{\op{b}} (\mathcal C) \\
a & \mapsto \op{tot}(\gamma(\theta(a))
\end{align*}
 where $\op{tot}$ denotes the total complex of a double complex over $\mathcal C$, is factored.
\end{lemma}

\begin{proof}
This is immediate from the definitions and the way that totalization extends the autoequivalences and natural transformations.
\end{proof}

\begin{corollary} \label{cor: composition of derived functors}
 Consider three triples, $(\mathcal A, \Phi, w), (\mathcal B, \Psi, v)$ and $(\mathcal C, \Upsilon, u)$
 \begin{align*}
  \theta &: \mathcal A \to \op{Ch}^{\op{b}}(\mathcal B) \\
  \gamma &: \mathcal B \to \op{Ch}^{\op{b}}(\mathcal C)
 \end{align*}
 be left exact factored functors. Suppose that there are enough $\theta$-adapted factorizations whose images are $\gamma$-adapted. Then, one has a natural isomorphism of derived factored functors
 \[
  \mathbf{R}\gamma^f \circ \mathbf{R}\theta^f \cong \mathbf{R}(\op{tot}(\gamma \circ \theta))^f.
 \]
 
 If we replace the assumption that that $\theta$ and $\gamma$ are left exact with the assumption that they are right exact, we have a natural isomorphism
 \begin{displaymath}
  \mathbf{L}\gamma^f \circ \mathbf{L}\theta^f \cong \mathbf{L}(\op{tot}(\gamma \circ \theta))^f.
 \end{displaymath}
\end{corollary}

\begin{proof}
 We can use $\theta$-adapted factorizations whose image under $\theta$ are $\gamma$-adapted. Plugging these in we are left with checking that
 \begin{displaymath}
  \gamma^f \circ \theta^f \cong \left( \op{tot} ( \gamma \circ \theta ) \right)^f.
 \end{displaymath}
 The difference between the two sides is the order of totalization, which does not matter up to isomorphism. 
\end{proof}

\begin{definition}
Let $\theta, \gamma$ be factored functors.  A \newterm{factored natural transformation} is a natural transformation
\[
\eta : \theta \to \gamma
\]
such that
\[
\Phi(\eta_A) = \eta_{\Phi(A)} 
\]
for any $A \in \mathcal A$.
\end{definition}

\begin{lemma} \label{lem: natural}
 Let $\eta: \theta \to \gamma$ be a factored natural transformation. Then there is a natural transformation
 \begin{align*}
  \eta^f & : \theta^f \to \gamma^f \\
  \eta^f_E & := (\eta_{E^{-1}}, \eta_{E^0})
 \end{align*}
 and an induced natural transformations on derived factored functors (if they exist)
 \begin{align*}
  \mathbf{R}\eta^f & : \mathbf{R}\theta^f \to \mathbf{R}\gamma^f \\
  \mathbf{L}\eta^f & : \mathbf{L}\theta^f \to \mathbf{L}\gamma^f.
 \end{align*}
\end{lemma}

\begin{proof}
It suffices to check that we have a natural transformation 
\begin{align*}
 \eta^f & : \theta^f \to \gamma^f \\
 \eta^f_E & := (\eta_{E^{-1}}, \eta_{E^0})
\end{align*}
as the natural transformation between the derived functors will come from restriction to injective factorizations. 

One easily checks that the dashed arrows in the following diagram can be filled in with the natural morphisms
\[
\eta_{\Phi^{-1}(E^0)}, \eta_{E^{-1}}, \eta_{E^0}, \eta_{\Phi^{-1}(F^0)}, \eta_{F^{-1}}, \eta_{F^0}.
\]
The requirement for the diagram to commute is equivalent to $\eta$ being factored.

Commutativity of this diagram exactly means that these morphisms define a natural transformation with the components lying in $\op{Ch}^b(\mathcal B)$.  Since $\op{tot}$ is a functor, we may apply it to the entire commutative diagram, to obtain such a diagram in $\mathcal B$.  The commutativity of the induced diagram in $\mathcal B$ is equivalent to verifying that these maps induce a natural transformation.

\begin{tikzpicture}[scale=0.8, every node/.style={scale=0.8},
back line/.style={densely dotted},is equiv
cross line/.style={preaction={draw=white, -,
line width=6pt}}]
\matrix (m) [ampersand replacement=\&, matrix of math nodes,
row sep=3em, column sep=3em,
text height=1.5ex,
text depth=0.25ex]{
\& \& \Phi^{-1}(\theta(E^0)) \& \& \theta(F^0) \&\\
\&\theta(E^{-1}) \& \& \theta(F^{-1}) \& \&\\
\theta(E^0) \& \&  \theta(F^0) \& \& \&\\
\& \& \& \Phi^{-1}(\gamma(E^0)) \& \& \Phi^{-1}(\gamma(F^0)) \\
\& \&\gamma(E^{-1}) \& \& \gamma(F^{-1}) \&\\
\& \gamma(E^0) \& \&  \gamma(F^0) \& \&\\
};
\path[->]
(m-1-3) edge node [auto] {$ \theta(\phi^{-1}(g^0)) $} (m-1-5)
(m-1-3) edge node [left, above] {$ \theta(\phi_E^{-1})\,\,\,\,\,\,\, $} (m-2-2)
(m-1-5) edge node [auto] {$ \theta(\phi_F^{-1}) $} (m-2-4)
(m-2-2) edge node [auto] {$ \theta(g^1) $}  (m-2-4)
(m-2-2) edge node [left, above] {$ \theta(\phi_E^{0}) \,\,\,\,\,\,\,$} (m-3-1)
(m-2-4) edge node [auto] {$ \theta(\phi_F^{0}) $} (m-3-3)
(m-3-1) edge node [auto] {$ \theta(g^0) $}  (m-3-3)
(m-4-4) edge node [auto] {$ \gamma(\phi^{-1}(g^0)) $}  (m-4-6)
(m-4-4) edge node [left, above] {$ \gamma(\phi_E^{-1}) \,\,\,\,\,\,\,$} (m-5-3)
(m-4-6) edge node [auto] {$ \gamma(\phi_F^{-1}) $} (m-5-5)
(m-5-3) edge node [auto] {$ \gamma(g^1) $} (m-5-5)
(m-5-3) edge node [left, above] {$ \gamma(\phi_E^{0}) \,\,\,\,\,\,\,$} (m-6-2)
(m-5-5) edge node [auto] {$ \gamma(\phi_F^{0}) $} (m-6-4)
(m-6-2) edge node [auto] {$ \gamma(g^0) $}  (m-6-4)
(m-1-3) edge [back line] (m-4-4)
(m-2-2) edge [back line]  (m-5-3)
(m-3-1) edge [back line] (m-6-2)
(m-1-5) edge [back line]  (m-4-6)
(m-2-4) edge [back line]  (m-5-5)
(m-3-3) edge [back line]  (m-6-4);
\end{tikzpicture}
\end{proof}

\begin{lemma} \label{lem: equal functors}
 Assume that $\mathcal A$ has enough injectives and that coproducts of injectives remain injective.
 Let
 \[
  \theta, \gamma : \mathcal B \to \op{Ch}^{\op{b}} (\mathcal C)
 \]	
 be left exact and let
 \[
  \eta: \theta \to \gamma
 \]
 be a factored natural transformation which induces a natural isomorphism
 \[
  \mathbf{R}\theta \cong \mathbf{R}\gamma,
 \]
 then $\mathbf{R}\eta^f$ induces a natural isomorphism,
 \[
  \mathbf{R}\theta^f \cong \mathbf{R}\gamma^f : \dcoFact{w} \to \dcoFact{v}.
 \]
 
 Assume that $\mathcal A$ has enough projectives and that products of projectives remain projectives.
 Let
 \[
  \theta, \gamma : \mathcal B \to \op{Ch}^{\op{b}} (\mathcal C)
 \]	
 be right exact and let
 \[
  \eta: \theta \to \gamma
 \]
 be a factored natural transformation which induces a natural isomorphism
 \[
  \mathbf{L}\theta \cong \mathbf{L}\gamma,
 \]
 then $\mathbf{L}\eta^f$ induces a natural isomorphism,
 \[
  \mathbf{L}\theta^f \cong \mathbf{L}\gamma^f : \dconFact{w} \to \dconFact{v}.
 \]
 
 Assume that $\mathcal A$ has finite injective/projective dimension. Let
 \[
  \theta, \gamma : \mathcal B \to \op{Ch}^{\op{b}} (\mathcal C)
 \]	
 be right/left exact and let
 \[
  \eta: \theta \to \gamma
 \]
 be a factored natural transformation which induces a natural isomorphism
 \begin{align*}
  \mathbf{R}\theta \cong \mathbf{R}\gamma / \mathbf{L}\theta \cong \mathbf{L}\gamma
 \end{align*}
 then $\mathbf{R}\eta^f/\mathbf{L}\eta^f$ induces a natural isomorphism,
 \[
  \mathbf{R}\theta^f \cong \mathbf{R}\gamma^f/\mathbf{L}\theta^f \cong \mathbf{L}\gamma^f : \dabsFact{w} \to \dabsFact{v}
 \] 
\end{lemma}

\begin{proof}
 We prove the first statement as the rest are extremely similar and follow from analogous arguments. By Theorem~\ref{theorem: strictification injective}, we may restrict our attention to factorizations $I$ folding pairs of bounded below injective complexes
 \begin{align*}
  0 & \to I_0^{-1} \to I^{-1}_1 \to \cdots \\
  0 & \to I_0^0 \to I^0_1 \to \cdots.
 \end{align*}
 Applying  $\eta: \theta \to \gamma$ to each of these complexes yields a quasi-isomorphism by assumption. Thus, the two complexes $\op{Cone}(\eta_{I_{\bullet}^{-1}}), \op{Cone}(\eta_{I_{\bullet}^{0}})$ are acyclic in the usual sense, they have no homology. Since the components of $\eta_{p,q}$ vanish for $p \not = q$, we can apply Lemma~\ref{lemma: cones over foldings are folded} to conclude that
 \begin{displaymath}
  \eta^f_I: \theta^f(I) \to \gamma^f(I)
 \end{displaymath}
 folds the two complexes $\op{Cone}(\eta_{I_{\bullet}^{-1}}), \op{Cone}(\eta_{I_{\bullet}^{0}})$. By Lemma~\ref{lemma: acyclic folds to acyclic}, $\op{Cone}(\eta^f_I)$ is co-acylic hence $\eta^f_I$ is a co-quasi-isomorphism.
\end{proof}

\begin{corollary} \label{cor: equivalences}
Suppose that $\mathcal A, \mathcal B$ have small coproducts and enough injectives, and that coproducts of injectives are injective. Let 
\[
 \theta :  \mathcal A \to  \mathcal B
\]
and
\[
 \gamma : \mathcal B \to  \mathcal A,
\]
be left-exact factored functors that commute with coproducts. Assume that there are factored natural transformations
\begin{displaymath}
 \eta: \op{Id}_{\mathcal A} \to \gamma \circ \theta, \delta: \op{Id}_{\mathcal B} \to \theta \circ \gamma. 
\end{displaymath}
inducing isomorphisms at the derived level. 

Assume further that there are enough $\theta$-adapted objects whose image under $\theta$ is $\gamma$-adapted and that there are enough $\gamma$-adapted objects whose image under $\gamma$ is $\theta$-adapted. Then the right derived factored functor
\[
 \mathbf{R}\theta^f: \dcoFact{w} \to \dcoFact{v}.
\]
is an equivalence.
\end{corollary}

\begin{proof}
Consider the factored functor
\[
 (\gamma \circ \theta)^f : \mathcal{A} \to \mathcal A
\]
By Corollary~\ref{cor: composition of derived functors} there is a natural isomorphism,
\[
\mathbf{R}(\gamma \circ \theta)^f \cong \mathbf{R}\gamma^f \circ \mathbf{R}\theta^f : \dcoFact{w} \to \dcoFact{w}.
\]
Now,
$\mathbf{R} (\gamma \circ \theta)$ is naturally isomorphic to the identity on $\op{D}(A)$ by assumption.  Hence, by Lemma~\ref{lem: equal functors}, the right derived factor functor $\mathbf{R} (\gamma \circ \theta)^f$ is naturally isomorphic to the identity on $\dcoFact{w}.$

Similarly, $\mathbf{R} (\theta \circ \gamma)^f$ is naturally isomorphic to the identity on $\dcoFact{v}$.
\end{proof}

\begin{corollary}
Suppose that $\mathcal A, \mathcal B$ have small products and enough projectives, and that products of projectives are projective. Let 
\[
\theta :  \mathcal A \to  \mathcal B
\]
and
\[
\gamma : \mathcal B \to  \mathcal A,
\]
be right-exact factored functors which commute with products. Assume that there are factored natural transformations
\begin{displaymath}
 \eta: \op{Id}_{\mathcal A} \to \gamma \circ \theta, \delta: \op{Id}_{\mathcal B} \to \theta \circ \gamma. 
\end{displaymath}
inducing isomorphisms at the derived level. 

Assume further that there are enough $\theta$-adapted objects whose image under $\theta$ is $\gamma$-adapted and that there are enough $\gamma$-adapted objects whose image under $\gamma$ is $\theta$-adapted. Then the right derived factored functor
\[
 \mathbf{L}\theta^f: \dcoFact{w} \to \dcoFact{v}.
\]
is an equivalence.
\end{corollary}

\begin{proof}
 We may replace $\mathcal A$ and $\mathcal B$ by their opposite categories and apply Corollary~\ref{cor: equivalences}
\end{proof}

%We can also make such a statement for fully-faithful functors.  In applications it is better to have such a statement for the objects themselves.  The proof will be more involved.

We now extend these statements to fully-faithful functors, in somewhat greater generality.

\begin{lemma} \label{lemma: ff left-exact}
 Assume that $\mathcal B$ has finite injective dimension. Let $E$ and $F$ be objects of $\op{Fact}(w)$ whose components have finite injective dimension.  If the map,
 \begin{displaymath}
  \mathbf{R}\theta: \op{Hom}_{\op{D}(\mathcal A)}(E^i,F^j[t]) \to \op{Hom}_{\op{D}(\mathcal B)}(\mathbf{R}\theta(E^i),\mathbf{R}\theta(F^j)[t]),
 \end{displaymath}
 is an isomorphism for all $i,j,t \in \Z$, then the map,
 \begin{displaymath}
  \mathbf{R}\theta^f: \op{Hom}_{\dcoFact{w}}(E,F[t]) \to \op{Hom}_{\dcoFact{v}}(\mathbf{R}\theta^f(E),\mathbf{R}\theta^f(F)[t]),
 \end{displaymath}
 is an isomorphism for all $t \in \Z$ 
\end{lemma}

The proof of Lemma \ref{lemma: ff left-exact} will be a direct result of studying a spectral sequence associated to a filtration on morphism complexes, $\textbf{Hom}^*$. Before presenting it, let us first recall one method for computing the maps,
\begin{displaymath}
 \mathbf{R}\theta: \op{Hom}_{\op{D}(\mathcal A)}(A,A') \to \op{Hom}_{\op{D}(\mathcal B)}(\mathbf{R}\theta(A'), \mathbf{R}\theta(A)),
\end{displaymath}
on the ordinary derived categories. 

Let $C,D$ be chain complexes from $\mathcal A$. We have the chain complex,
\begin{displaymath}
 \textbf{Hom}_{\mathcal A}^n(C,D) = \prod_{j-i=n} \op{Hom}_{\mathcal A}(C^i,D^j), 
\end{displaymath}
with 
\begin{displaymath}
 d( \prod_i g^i: C^i \to D^{i+n} ) := \prod_i (d^{i+n+1}_D \circ g^i - (-1)^n g^{i+1} \circ d_C^{i+1}).
\end{displaymath}

First, we choose injective resolutions,
\begin{align*}
 0 & \to A' \overset{d_0'}{\to} I_0' \overset{d_1'}{\to} I_1' \overset{d_2'}{\to} I_2' \overset{d_3'}{\to} \cdots \\
 0 & \to A \overset{d_0}{\to} I_0 \overset{d_1}{\to} I_1 \overset{d_2}{\to} I_2 \overset{d_3}{\to} \cdots.
\end{align*}
Next, we construct a commutative diagram of bounded complexes,
\begin{center}
\begin{tikzpicture}[description/.style={fill=white,inner sep=2pt}]
 \matrix (m) [matrix of math nodes, row sep=3em, column sep=3em, text height=1.5ex, text depth=0.25ex]
 { \theta(I_0) & \theta(I_1) & \theta(I_2) & \theta(I_3) & \cdots \\
   J_{0,0} & J_{1,0} & J_{2,0} & J_{3,0} & \cdots \\ 
   J_{0,1} & J_{1,1} & J_{2,1} & J_{3,1} & \cdots \\ 
   J_{0,2} & J_{1,2} & J_{2,2} & J_{3,2} & \cdots \\
   \vdots & \vdots & \vdots & \vdots &  \\ };
 \path[->,font=\scriptsize]
 (m-1-1) edge node[above]{$\theta(d_1)$} (m-1-2) 
 (m-1-2) edge node[above]{$\theta(d_2)$} (m-1-3) 
 (m-1-3) edge node[above]{$\theta(d_3)$} (m-1-4)
 (m-1-4) edge node[above]{$\theta(d_4)$} (m-1-5) 
 
 (m-1-1) edge node[left]{$d_{0,0}^v$} (m-2-1) 
 (m-1-2) edge node[left]{$d_{1,0}^v$} (m-2-2) 
 (m-1-3) edge node[left]{$d_{2,0}^v$} (m-2-3)
 (m-1-4) edge node[left]{$d_{3,0}^v$} (m-2-4) 
 
 (m-2-1) edge node[above]{$d^h_{1,0}$} (m-2-2) 
 (m-2-2) edge node[above]{$d^h_{2,0}$} (m-2-3) 
 (m-2-3) edge node[above]{$d^h_{3,0}$} (m-2-4)
 (m-2-4) edge node[above]{$d^h_{4,0}$} (m-2-5) 
 
 (m-2-1) edge node[left]{$d_{0,1}^v$} (m-3-1) 
 (m-2-2) edge node[left]{$d_{1,1}^v$} (m-3-2) 
 (m-2-3) edge node[left]{$d_{2,1}^v$} (m-3-3)
 (m-2-4) edge node[left]{$d_{3,1}^v$} (m-3-4) 
 
 (m-3-1) edge node[above]{$d^h_{1,1}$} (m-3-2) 
 (m-3-2) edge node[above]{$d^h_{2,1}$} (m-3-3) 
 (m-3-3) edge node[above]{$d^h_{3,1}$} (m-3-4)
 (m-3-4) edge node[above]{$d^h_{4,1}$} (m-3-5) 
 
 (m-3-1) edge node[left]{$d_{0,2}^v$} (m-4-1) 
 (m-3-2) edge node[left]{$d_{1,2}^v$} (m-4-2) 
 (m-3-3) edge node[left]{$d_{2,2}^v$} (m-4-3)
 (m-3-4) edge node[left]{$d_{3,2}^v$} (m-4-4) 
 
 (m-4-1) edge node[above]{$d^h_{1,2}$} (m-4-2) 
 (m-4-2) edge node[above]{$d^h_{2,2}$} (m-4-3) 
 (m-4-3) edge node[above]{$d^h_{3,2}$} (m-4-4)
 (m-4-4) edge node[above]{$d^h_{4,2}$} (m-4-5) 
 
 (m-4-1) edge node[left]{$d_{0,3}^v$} (m-5-1) 
 (m-4-2) edge node[left]{$d_{1,3}^v$} (m-5-2) 
 (m-4-3) edge node[left]{$d_{2,3}^v$} (m-5-3)
 (m-4-4) edge node[left]{$d_{3,3}^v$} (m-5-4) 
 ;
\end{tikzpicture} 
\end{center}
where the rows and columns are exact complexes of bounded complexes, all squares commute, and each $J_{p,q}$ is a bounded complex of injectives. Form the associated total complex, $J=(J_{\bullet},d^{tot}_{\bullet})$, where
\begin{displaymath}
 J_u = \bigoplus_{r+s+t = u} J_{r,s,t}
\end{displaymath}
where $J_{r,s,t}= (J_{s,t})_r$, with the differential,
\begin{displaymath}
 d^{tot}_u: J_u \to J_u \oplus J_{u+1},
\end{displaymath}
is the product of the maps,
\begin{displaymath}
 J_{r,s,t} \overset{d_{r}^{J_{s,t}} \oplus (-1)^r d_{s+1,t}^h \oplus (-1)^{r+s} d^v_{s,t+1}}{\to} J_{r+1,s,t} \oplus J_{r,s+1,t} \oplus J_{r,s,t+1}.
\end{displaymath}
This comes with a map of chain complexes, $\op{tot}\theta(I) \to J$.

The map,\begin{displaymath}
 \mathbf{R}\theta: \op{Hom}_{\op{D}(\mathcal A)}(A',A[t]) \to \op{Hom}_{\op{D}(\mathcal B)}(\mathbf{R}\theta(A'), \mathbf{R}\theta(A)[t]),
\end{displaymath}
is the cohomology in degree $t$ of the map of chain complexes,
\begin{displaymath}
 \textbf{Hom}_{\mathcal A}^*(I',I) \overset{\theta}{\to} \textbf{Hom}_{\mathcal A}^*(\op{tot}\theta(I'),\op{tot}\theta(I)) \to \textbf{Hom}_{\mathcal A}^*(\op{tot}\theta(I'),J).
\end{displaymath}
With this recap fresh in our mind, let us proceed with the proof of Lemma \ref{lemma: ff left-exact}.

\begin{proof}[Proof of Lemma \ref{lemma: ff left-exact}]
 Choose finite injective resolutions of the components,
 \begin{align*}
  0 & \to E^{-1} \overset{d_0^{E^{-1}}}{\to} I_0^{E^{-1}} \overset{d_1^{E^{-1}}}{\to} I_1^{E^{-1}} \overset{d_2^{E^{-1}}}{\to} I_2^{E^{-1}} \overset{d_3^{E^{-1}}}{\to} \cdots \\
  0 & \to E^{0} \overset{d_0^{E^0}}{\to} I_0^{E^0} \overset{d_1^{E^0}}{\to} I_1^{E^0} \overset{d_2^{E^0}}{\to} I_2^{E^0} \overset{d_3^{E^0}}{\to} \cdots \\
  0 & \to F^{-1} \overset{d_0^{F^{-1}}}{\to} I_0^{F^{-1}} \overset{d_1^{F^{-1}}}{\to} I_1^{F^{-1}} \overset{d_2^{F^{-1}}}{\to} I_2^{F^{-1}} \overset{d_3^{F^{-1}}}{\to} \cdots \\
  0 & \to F^{0} \overset{d_0^{F^0}}{\to} I_0^{F^0} \overset{d_1^{F^0}}{\to} I_1^{F^0} \overset{d_2^{F^0}}{\to} I_2^{F^0} \overset{d_3^{F^0}}{\to} \cdots,
 \end{align*}
 and apply Theorem \ref{theorem: strictification injective} to get resolutions of factorizations,
 \begin{align*}
  E & \to I^E \\
  F & \to I^F.
 \end{align*}
 
 Recall that the components of $I^F$ are
 \begin{align*}
  (I^F)^{-1} & = \bigoplus_{2l} \Phi^{-l}(I^{-1}_{2l}) \oplus \bigoplus_{2l+1} \Phi^{-l-1}(I^{0}_{2l+1}) \\
  (I^F)^{0} & = \bigoplus_{2l} \Phi^{-l}(I^{0}_{2l}) \oplus \bigoplus_{2l+1} \Phi^{-l}(I^{-1}_{2l+1}).
 \end{align*}
 Applying $\theta^f$, we get the factorization, $\theta^f(I^F)$, whose components are
 \begin{align*}
  \theta^f(I^F)^{-1} & = \bigoplus_{2l} \theta^f(\Phi^{-l}(I^{-1}_{2l})) \oplus \bigoplus_{2l+1} \theta^f(\Phi^{-l-1}(I^{0}_{2l+1})) \\
  \theta^f(I^F)^{0} & = \bigoplus_{2l} \theta^f(\Phi^{-l}(I^{0}_{2l})) \oplus \bigoplus_{2l+1} \theta^f(\Phi^{-l}(I^{-1}_{2l+1})).
 \end{align*}
 We want to replace $\theta^f(I^F)$ by an injective factorization to compute $\mathbf{R}\theta^f$.  We will apply Theorem \ref{theorem: strictification injective}, but, first, we need to choose injective resolutions of the components of $\theta^f(I^F)$. 
 
 To do this, we first construct finite diagrams (which exist by assumption),
 \begin{center}
 \begin{tikzpicture}[description/.style={fill=white,inner sep=2pt}]
  \matrix (m) [matrix of math nodes, row sep=3em, column sep=3em, text height=1.5ex, text depth=0.25ex]
  { \theta^f(I_0^{F^{-1}}) & \theta^f(I_1^{F^{-1}}) & \theta^f(I_2^{F^{-1}}) & \cdots \\
    J_{0,0}^{F^{-1}} & J_{1,0}^{F^{-1}} & J_{2,0}^{F^{-1}} & \cdots \\ 
    J_{0,1}^{F^{-1}} & J_{1,1}^{F^{-1}} & J_{2,1}^{F^{-1}} & \cdots \\ 
    \vdots & \vdots & \vdots &  \\ };
  \path[->,font=\scriptsize]
  (m-1-1) edge node[above]{$\theta(d_1^{F^{-1}})$} (m-1-2) 
  (m-1-2) edge node[above]{$\theta(d_2^{F^{-1}})$} (m-1-3) 
  (m-1-3) edge node[above]{$\theta(d_3^{F^{-1}})$} (m-1-4)
 
  (m-1-1) edge node[left]{$d_{0,0}^{v,F^{-1}}$} (m-2-1) 
  (m-1-2) edge node[left]{$d_{1,0}^{v,F^{-1}}$} (m-2-2) 
  (m-1-3) edge node[left]{$d_{2,0}^{v,F^{-1}}$} (m-2-3)
 
  (m-2-1) edge node[above]{$d^{h,F^{-1}}_{1,0}$} (m-2-2) 
  (m-2-2) edge node[above]{$d^{h,F^{-1}}_{2,0}$} (m-2-3) 
  (m-2-3) edge node[above]{$d^{h,F^{-1}}_{3,0}$} (m-2-4)
 
  (m-2-1) edge node[left]{$d_{0,1}^{v,F^{-1}}$} (m-3-1) 
  (m-2-2) edge node[left]{$d_{1,1}^{v,F^{-1}}$} (m-3-2) 
  (m-2-3) edge node[left]{$d_{2,1}^{v,F^{-1}}$} (m-3-3) 
 
  (m-3-1) edge node[above]{$d^{h,F^{-1}}_{1,1}$} (m-3-2) 
  (m-3-2) edge node[above]{$d^{h,F^{-1}}_{2,1}$} (m-3-3) 
  (m-3-3) edge node[above]{$d^{h,F^{-1}}_{3,1}$} (m-3-4)
 
  (m-3-1) edge node[left]{$d_{0,2}^{v,F^{-1}}$} (m-4-1) 
  (m-3-2) edge node[left]{$d_{1,2}^{v,F^{-1}}$} (m-4-2) 
  (m-3-3) edge node[left]{$d_{2,2}^{v,F^{-1}}$} (m-4-3)

  ;
  \end{tikzpicture} 
 \end{center}
 and
 \begin{center}
 \begin{tikzpicture}[description/.style={fill=white,inner sep=2pt}]
  \matrix (m) [matrix of math nodes, row sep=3em, column sep=3em, text height=1.5ex, text depth=0.25ex]
  { \theta^f(I_0^{F^{-1}}) & \theta^f(I_1^{F^{-1}}) & \theta^f(I_2^{F^{-1}}) & \cdots \\
    J_{0,0}^{F^{-1}} & J_{1,0}^{F^{-1}} & J_{2,0}^{F^{-1}} & \cdots \\ 
    J_{0,1}^{F^{-1}} & J_{1,1}^{F^{-1}} & J_{2,1}^{F^{-1}} & \cdots \\
    \vdots & \vdots & \vdots &  \\ };
  \path[->,font=\scriptsize]
  (m-1-1) edge node[above]{$\theta(d_1^{F^{0}})$} (m-1-2) 
  (m-1-2) edge node[above]{$\theta(d_2^{F^0})$} (m-1-3) 
  (m-1-3) edge node[above]{$\theta(d_3^{F^0})$} (m-1-4)
 
  (m-1-1) edge node[left]{$d_{0,0}^{v,F^0}$} (m-2-1) 
  (m-1-2) edge node[left]{$d_{1,0}^{v,F^0}$} (m-2-2) 
  (m-1-3) edge node[left]{$d_{2,0}^{v,F^0}$} (m-2-3)
 
  (m-2-1) edge node[above]{$d^{h,F^0}_{1,0}$} (m-2-2) 
  (m-2-2) edge node[above]{$d^{h,F^0}_{2,0}$} (m-2-3) 
  (m-2-3) edge node[above]{$d^{h,F^0}_{3,0}$} (m-2-4)
 
  (m-2-1) edge node[left]{$d_{0,1}^{v,F^0}$} (m-3-1) 
  (m-2-2) edge node[left]{$d_{1,1}^{v,F^0}$} (m-3-2) 
  (m-2-3) edge node[left]{$d_{2,1}^{v,F^0}$} (m-3-3)  
 
  (m-3-1) edge node[above]{$d^{h,F^0}_{1,1}$} (m-3-2) 
  (m-3-2) edge node[above]{$d^{h,F^0}_{2,1}$} (m-3-3) 
  (m-3-3) edge node[above]{$d^{h,F^0}_{3,1}$} (m-3-4)
  (m-3-1) edge node[left]{$d_{0,2}^{v,F^0}$} (m-4-1) 
  (m-3-2) edge node[left]{$d_{1,2}^{v,F^0}$} (m-4-2) 
  (m-3-3) edge node[left]{$d_{2,2}^{v,F^0}$} (m-4-3)
%  
%   (m-4-1) edge node[above]{$d^{h,F^0}_{1,2}$} (m-4-2) 
%   (m-4-2) edge node[above]{$d^{h,F^0}_{2,2}$} (m-4-3) 
%   (m-4-3) edge node[above]{$d^{h,F^0}_{3,2}$} (m-4-4)
%   (m-4-4) edge node[above]{$d^{h,F^0}_{4,2}$} (m-4-5) 
%  
%   (m-4-1) edge node[left]{$d_{0,3}^{v,F^0}$} (m-5-1) 
%   (m-4-2) edge node[left]{$d_{1,3}^{v,F^0}$} (m-5-2) 
%   (m-4-3) edge node[left]{$d_{2,3}^{v,F^0}$} (m-5-3)
%   (m-4-4) edge node[left]{$d_{3,3}^{v,F^0}$} (m-5-4) 
  ;
  \end{tikzpicture} 
 \end{center}
 where the rows and columns are exact, all $J$'s are bounded complexes of injectives, and all squares commute. 
 
 Then, we use the injective resolutions,
 \begin{align*}
  0 \to \theta^f(I^F)^{-1} \overset{\bigoplus_{2l} \Psi^{-l}(d_{0,2l}^{v,F^{-1}}) \oplus \bigoplus_{2l+1} \Psi^{-l-1}(d_{0,2l+1}^{v,F^{0}})}{\to} \bigoplus_{2l} \Psi^{-l}(J^{F^{-1}}_{0,2l}) \oplus \bigoplus_{2l+1} \Psi^{-l-1}(J^{F^{0}}_{0,2l+1}) \\ \overset{\bigoplus_{2l} \Psi^{-l}(d_{1,2l}^{v,F^{-1}}) \oplus \bigoplus_{2l+1} \Psi^{-l-1}(d_{1,2l+1}^{v,F^{0}})}{\to}  \bigoplus_{2l} \Psi^{-l}(J^{F^{-1}}_{1,2l}) \oplus \bigoplus_{2l+1} \Psi^{-l-1}(J^{F^{0}}_{1,2l+1}) \to  \cdots \\
 \end{align*}
 and
 \begin{align*}
  0 \to \theta^f(I^F)^{0} \overset{\bigoplus_{2l} \Psi^{-l}(d_{0,2l}^{v,F^{0}}) \oplus \bigoplus_{2l+1} \Psi^{-l}(d_{0,2l+1}^{v,F^{-1}})}{\to} \bigoplus_{2l} \Psi^{-l}(J^{F^{0}}_{0,2l}) \oplus \bigoplus_{2l+1} \Psi^{-l}(J^{F^{-1}}_{0,2l+1}) \\ \overset{\bigoplus_{2l} \Psi^{-l}(d_{1,2l}^{v,F^{0}}) \oplus \bigoplus_{2l+1} \Psi^{-l}(d_{1,2l+1}^{v,F^{-1}})}{\to}  \bigoplus_{2l} \Psi^{-l}(J^{F^{0}}_{1,2l}) \oplus \bigoplus_{2l+1} \Psi^{-l}(J^{F^{-1}}_{1,2l+1}) \to  \cdots \\
 \end{align*}
 totalize and apply Theorem \ref{theorem: strictification injective}. Denote the resulting factorization by $J$. Note that the components of $J$ are
 \begin{align*}
  J^{-1} & = \bigoplus_{r+s+t=2l} \Psi^{-l}(J^{F^{-1}}_{r,s,t}) \oplus \bigoplus_{r+s+t=2l+1} \Psi^{-l-1}(J^{F^0}_{r,s,t}) \\
  J^{0} & = \bigoplus_{r+s+t=2l} \Psi^{-l}(J^{F^0}_{r,s,t}) \oplus \bigoplus_{r+s+t=2l+1} \Psi^{-l}(J^{F^{-1}}_{r,s,t}).
 \end{align*}

 The chain complex, $\textbf{Hom}_v^*(\theta^f(I^E),J)$, admits a filtration,
 \begin{gather*}
  \mathcal F^p\textbf{Hom}_v^n(\theta^f(I^E),J) := \{ (g^{-1},g^0) \mid \forall q \\ 
  \begin{cases} 
  g^{-1} \left(\bigoplus_{u+v = 2l \leq q-1} \Psi^{-l}(\theta^f(I^{E^{-1}}_{v})_u) \oplus \bigoplus_{u+v = 2l+1 \leq q-1} \Psi^{-l-1}(\theta^f(I^{E^{0}}_{v})_u) \right) \\ \subseteq \bigoplus_{r+s+t = 2l \leq q+p+n-1} \Psi^{m-l}(J_{r,s,t}^{F^{-1}}) \oplus \bigoplus_{r+s+t = 2l+1 \leq q+p+n-1} \Psi^{m-l-1}(J_{r,s,t}^{F^{0}}) \\ 
  g^{0} \left(\bigoplus_{u+v = 2l \leq q} \Psi^{-l}(\theta^f(I^{E^{0}}_{v})_u) \oplus \bigoplus_{u+v = 2l+1 \leq q} \Psi^{-l}(\theta^f(I^{E^{-1}}_{v})_u) \right) \\ \subseteq \bigoplus_{r+s+t = 2l \leq q+p+n} \Psi^{m-l}(J_{r,s,t}^{F^{0}}) \oplus \bigoplus_{r+s+t = 2l+1 \leq t+q+n} \Psi^{m-l}(J_{r,s,t}^{F^{-1}}) \} & n=2m \\
  g^{-1} \left(\bigoplus_{u+v = 2l \leq q-1} \Psi^{-l}(\theta^f(I^{E^{-1}}_{v})_u) \oplus \bigoplus_{u+v = 2l+1 \leq q-1} \Psi^{-l-1}(\theta^f(I^{E^{0}}_{v})_u) \right) \\ \subseteq \bigoplus_{r+s+t = 2l \leq q+p+n-1} \Psi^{m-l}(J_{r,s,t}^{F^{0}}) \oplus \bigoplus_{r+s+t = 2l+1 \leq q+p+n-1} \Psi^{m-l}(J_{r,s,t}^{F^{-1}}) \\ 
  g^{0} \left(\bigoplus_{u+v = 2l \leq t} \Psi^{-l}(\theta^f(I^{E^{0}}_{v})_u) \oplus \bigoplus_{u+v = 2l+1 \leq t} \Psi^{-l}(\theta^f(I^{E^{-1}}_{v})_u) \right) \\ \subseteq \bigoplus_{r+s+t = 2l \leq q+p+n} \Psi^{m-l+1}(J_{r,s,t}^{F^{-1}}) \oplus \bigoplus_{r+s+t = 2l+1 \leq q+p+n} \Psi^{m-l}(J_{r,s,t}^{F^{0}}) \} & n=2m+1 
  \end{cases}
 \end{gather*}
 After recombining even and odd parts, the associated graded complex is
%  \begin{gather*}
%   \op{Gr}^p_{\mathcal F}\textbf{Hom}_v^n(\theta^f(I^E),J) = \\
%   \begin{cases} 
%    \prod_{t=2s} \op{Hom}_{\mathcal B}(\theta^f(I^{E^0}_{t-1}),\bigoplus_{r+s=n+p+t-1} \Psi^{-q}(J^{F^0}_{r,s})) \oplus \prod_{t=2s+1} \op{Hom}_{\mathcal B}(\theta^f(I^{E^{-1}}_{t-1}),\bigoplus_{r+s=n+p+t-1} \Psi^{-q}(J^{F^{-1}}_{r,s})) \\ \oplus \prod_{t=2s} \op{Hom}_{\mathcal B}(\theta^f(I^{E^0}_{t}),\bigoplus_{r+s=n+p+t} \Psi^{-q}(J^{F^0}_{r,s})) \oplus \prod_{t=2s+1} \op{Hom}_{\mathcal B}(\theta^f(I^{E^{-1}}_{t}),\bigoplus_{r+s=n+p+t} \Psi^{-q}(J^{F^{-1}}_{r,s})) & p=2q \\
%    \prod_{t=2s} \op{Hom}_{\mathcal B}(\theta^f(I^{E^0}_{t-1}),\bigoplus_{r+s=n+p+t-1} \Psi^{-q}(J^{F^{-1}}_{r,s})) \oplus \prod_{t=2s+1} \op{Hom}_{\mathcal B}(\theta^f(I^{E^{-1}}_{t-1}),\bigoplus_{r+s=n+p+t-1} \Psi^{-q-1}(J^{F^{0}}_{r,s})) \\ \oplus \prod_{t=2s} \op{Hom}_{\mathcal B}(\theta^f(I^{E^0}_{t}),\bigoplus_{r+s=n+p+t} \Psi^{-q}(J^{F^{-1}}_{r,s})) \oplus \prod_{t=2s+1} \op{Hom}_{\mathcal B}(\theta^f(I^{E^{-1}}_{t}),\bigoplus_{r+s=n+p+t} \Psi^{-q-1}(J^{F^{0}}_{r,s})) & p=2q+1
%   \end{cases}
%  \end{gather*}
%  which we can rewrite, by combining even and odd parts, as 
 \begin{gather*}
  \op{Gr}^p\textbf{Hom}_v^n(\theta^f(I^E),J) = \\
  \begin{cases} 
   \textbf{Hom}^{p+n}_{\mathcal B}(\op{tot}\theta^f(I^{E^0}),\Psi^{-q}(J^{F^0})) \oplus \textbf{Hom}^{p+n}_{\mathcal B}(\op{tot}\theta^f(I^{E^{-1}}),\Psi^{-q}(J^{F^{-1}})) & p=2q \\
   \textbf{Hom}^{p+n}_{\mathcal B}(\op{tot}\theta^f(I^{E^0}),\Psi^{-q}(J^{F^{-1}})) \oplus \textbf{Hom}^{p+n}_{\mathcal B}(\op{tot}\theta^f(I^{E^{-1}}),\Psi^{-q-1}(J^{F^{0}})) & p=2q+1
  \end{cases}
 \end{gather*}
 with the differential being the sum of the differentials on $\textbf{Hom}^{*}_{\mathcal B}(\op{tot}\theta^f(I^{E^{u}}),\Psi^{-q}(J^{F^v}))$, $u,v \in \{-1,0\}$.
 
 There exists an analogous filtration on $\textbf{Hom}_v^*(I^E,I^F)$ whose associated graded complex is 
 \begin{displaymath}
  \begin{cases} 
   \textbf{Hom}^{p+n}_{\mathcal A}(I^{E^0},\Phi^{-q}(I^{F^0})) \oplus \textbf{Hom}^{p+n}_{\mathcal A}(I^{E^{-1}},\Phi^{-q}(I^{F^{-1}})) & p=2q \\
   \textbf{Hom}^{p+n}_{\mathcal A}(I^{E^0}, \Phi^{-q}(I^{F^{-1}})) \oplus \textbf{Hom}^{p+n}_{\mathcal A}(I^{E^{-1}}, \Phi^{-q-1}(I^{F^{0}})) & p=2q+1
  \end{cases}
 \end{displaymath}
 These filtrations are compatible with the map,
 \begin{displaymath}
  \textbf{Hom}_v^*(I^E,I^F) \to \textbf{Hom}_v^*(\theta^f(I^E),J).
 \end{displaymath}
 The map on the associated graded complexes,
 \begin{displaymath}
  \op{Gr}^p\textbf{Hom}_v^*(I^E,I^F) \to \op{Gr}^p\textbf{Hom}_w^*(\theta^f(I^E),J^F),
 \end{displaymath}
 is exactly the sum of the maps,
 \begin{align*}
  \textbf{Hom}_{\mathcal A}^*(I^{E^0},\Phi^{-q}(I^{F^0})) & \to \textbf{Hom}_{\mathcal B}^*(\op{tot}\theta^f(I^{E^0}),\Psi^{-q}(J^{F^0})) \\
  \textbf{Hom}_{\mathcal A}^*(I^{E^0},\Phi^{-q}(I^{F^{-1}})) & \to \textbf{Hom}_{\mathcal B}^*(\op{tot}\theta^f(I^{E^0}),\Psi^{-q}(J^{F^{-1}})) \\
  \textbf{Hom}_{\mathcal A}^*(I^{E^{-1}},\Phi^{-q-1}(I^{F^0})) & \to \textbf{Hom}_{\mathcal B}^*(\op{tot}\theta^f(I^{E^{-1}}),\Psi^{-q-1}(J^{F^0})) \\
  \textbf{Hom}_{\mathcal A}^*(I^{E^{-1}},\Phi^{-q}(I^{F^{-1}})) & \to \textbf{Hom}_{\mathcal B}^*(\op{tot}\theta^f(I^{E^{-1}}),\Psi^{-q}(J^{F^{-1}})),
 \end{align*}
 which we have assumed to be quasi-isomorphisms. The corresponding map of spectral sequences is an isomorphism on the $E_1$-page. Since the injective resolutions are assumed to be finite, the spectral sequence degenerates and yields the desired statement.
\end{proof}

\begin{lemma} \label{lemma: ff right-exact}
 Let $E$ and $F$ be objects of $\op{Fact}(w)$. The map,
 \begin{displaymath}
  \mathbf{L}\theta^f: \op{Hom}_{\dconFact{w}}(E,F[t]) \to \op{Hom}_{\dconFact{v}}(\mathbf{L}\theta^f E,\mathbf{L}\theta^f F[t]),
 \end{displaymath}
 is an isomorphism for all $t\in \Z$ if the map
 \begin{displaymath}
  \mathbf{L}\theta: \op{Hom}_{\op{D}(\mathcal A)}(E^i,F^j[t]) \to \op{Hom}_{\op{D}(\mathcal B)}(\mathbf{L}\theta E^i,\mathbf{L}\theta F^j[t]),
 \end{displaymath}
 is an isomorphism for all $i,j,t \in \Z$.
\end{lemma}

\begin{proof}
 The proof of this lemma is completely analogous to the proof of Lemma \ref{lemma: ff left-exact}.
\end{proof}

\section{Geometric Applications}
In this section we apply our results to smooth varieties/stacks. The category of quasi-coherent sheaves on such a space has finite injective dimension.

\subsection{Complete Intersections as matrix factorizations over noncommutative algebras}

Let $X$ be a smooth algebraic variety over $k$, $\mathcal E$ be a vector bundle on $X$ and consider the associated geometric vector bundle
\[
 \op{V}(\mathcal E) := \op{Spec} \left( \op{Sym}^{\bullet} \mathcal E \right)
\]
together with a scaling action, $\mathbb{G}_m$, by the units of $k$. Suppose that $\op{V}(\mathcal E)$ admits a $\mathbb{G}_m$-equivariant tilting object $T$. Then $A := \op{End}(T)$ is a $\Z$-graded algebra. Let $\op{mod}_{\Z}A$ be the Abelian category of finitely generated graded modules over $A$.  Let $(1) : \op{mod}_{\Z}A \to \op{mod}_{\Z}A$ be the autoequivalence given by shifting the grading of a module
\[
M(1)_j := M_{j+1}.
\]
Now, let $s \in \op{H}^0(X,\mathcal E)$ be a global section and $Z$ be the zero locus of $s$ and consider the equivariant line bundle $\mathcal L$ obtained by pulling back the representation given by the scaling action on $k$. Then $s$ gives a map $T \to T \otimes \mathcal L$ and hence an element of $A$. We define a natural transformation
\[
v: \op{Id} \to (1)
\] 
given by the action of this element of $A$ on modules.

\begin{theorem}
Assume that the codimension of $Z$ equals the rank of $\mathcal E$. There is an equivalence of triangulated categories
\[
 \dbcoh{Z} \cong \op{K}(\op{Fact} \mathcal Proj , v).
\]
\end{theorem}

\begin{proof}
Using the triples $(\op{coh}_{\mathbb{G}_m}\op{V}(\mathcal E), ( - \otimes \mathcal L),  s)$ or $(\op{Qcoh}_{\mathbb{G}_m}\op{V}(\mathcal E), ( - \otimes \mathcal L),  s)$ we obtain a category of factorizations as in Example~\ref{ex: gauged LG model}. Since $\op{Qcoh}_{\mathbb{G}_m} \op{V}(\mathcal E)$ has enough injectives, we may assume that $T$ is a complex of injectives.
Then 
\[
\op{Hom}(T, -) : \op{coh}_{\mathbb{G}_m} \op{V}(\mathcal E) \to \op{Ch}^{\op{b}}(\op{Mod}_{\Z}A)
\]
is an exact factored functor which induces an equivalence
\[
\op{D}^{\op{b}}( \op{coh}_{\mathbb{G}_m} \op{V}(\mathcal E)) \to \op{D}^{\op{b}}(\op{mod}_{\Z}A).
\]
The inverse is also induced by a factored functor
\[
(- \otimes T) : \op{mod}_{\Z}A \to  \op{Ch}^{\op{b}}(\op{coh}_{\mathbb{G}_m}\op{V}(\mathcal E)).
\]
Note, we have slightly abused notation here, in that $\op{Ch}^{\op{b}}(\op{coh}_{\mathbb{G}_m}\op{V}(\mathcal E))$ is meant to be the full subcategory of $\op{Ch}(\op{Qcoh}_{\mathbb{G}_m}\op{V}(\mathcal E))$ with bounded and coherent cohomology.

%so that the natural morphism
%\[
%\op{Hom}(T,  -  \otimes T) \to ( - \otimes T) \to \op{Id}_{\mathcal B} 
%\]
%induces an natural isomorphism
%\[
%\mathbf{L}(\op{Hom}(T, - \otimes T) \to \op{Id}_{\dqcoh{B}}.
%\]
%such that the image of $\mathbf{L}(- \otimes I )$ is $\op{Hom}(T, -)$-adapted.  
There are natural isomorphisms
\begin{align*}
\mathbf{R}\op{Hom}(T, -)^f  \circ \ \mathbf{L}(- \otimes T )^f 
&  \cong \op{Hom}(T,  -)^f  \circ \mathbf{L}(- \otimes T )^f \\
&  \cong \mathbf{L}( \op{Hom}(T,  -)^f  \circ (- \otimes T )^f )\\
&  \cong \mathbf{L}(- \otimes \op{Hom}(T,T ) )^f\\
&  \cong \mathbf{L}\op{Id}^f\\
& \cong \op{Id}_{\dabsFact{v}}.
\end{align*}
The first line uses the fact that $\op{Hom}(T, -)$-adapted is exact since $T$ is a bounded complex of injectives.  The second line uses the definition of a left-derived factored functor. The third and fourth lines uses that the natural maps
\begin{displaymath}
 M \to M \otimes \op{Hom}(T,T) \to \op{Hom}(T, M \otimes T) 
\end{displaymath}
are isomorphisms for projective $\op{Hom}(T,T)$-modules and hence on factorizations with projective components, Lemma~\ref{lem: equal functors} The fifth line follows as $- \otimes \op{Hom}(T,T) \cong \op{Id}$.
 
Therefore, $\mathbf{R}\op{Hom}(T, -)^f$ is essentially surjective.  It is fully-faithful by Lemma~\ref{lemma: ff left-exact} i.e.
\begin{equation}
\dabsFact{w} \cong \dabsFact{v}
\label{eq: noncommutative equiv}
\end{equation}

We now have equivalences
\begin{align*}
 \dbcoh{Z} & \cong \dabsFact{w} \\
 & \cong \dabsFact{v}. \\
 & \cong \op{K}(\op{Fact} \mathcal Proj , v).
\end{align*}
The first line is precisely \cite[Theorem 3.4]{Shipman} or equivalently \cite[Theorem 3.6]{Isik}. The second line is Equation~\eqref{eq: noncommutative equiv}. The final line is Corollary~\ref{corollary: homotopy category of inj/proj is derived cat}.
\end{proof}

\begin{example}
Let $X = \P^n$ and $\mathcal E = \O(d)$.  Then $Z$ is a hypersurface defined by $s$, a homogeneous polynomial of degree $d$.  Let $\pi : \op{V}(\O(d)) \to \P^n$ be the projection.  We may consider the tilting object
\[
T = \pi^*(\O \oplus \cdots \oplus \O(n)).
\]
Let $R := k[x_0, \ldots, x_n]$ and denote by $R_m$ homogeneous polynomials of degree $m$.  
Then one easily verifies that
\[
A = \bigoplus_{1 \leq i, j \leq n} \bigoplus_{t=0}^\infty R_{td+i-j}
\]
with the algebra structure given by
\[
rr' = 
\begin{cases}
rr' & \tif i=j' \\
0 & \text{otherwise}
\end{cases}
\in R_{(t+t')d+i'-j}
\]
where $r \in R_{td+i-j}, r' \in R_{t'd+i'-j'}$.  Furthermore, an element of $R_{td+i+j}$ has degree $t$.  For $1 \leq h \leq n$, our homogeneous polynomial $s$ gives an element $v_h$ in the summand $R_d$ corresponding to $i=j=h$.  Then, 
\[
v := \sum_{h=1}^n v_h 
\]
is a central element of $A$ of degree $1$.  We may consider $v$ as a natural transformation $\op{Id} \to (1)$ where $(1)$ denotes the grading shift by $1$.  We get an equivalence
\[
\dbcoh{Z} \cong \op{K}(\op{Fact} \mathcal Proj , v).
\]
The right hand side is a the same as graded matrix factorizations of $(A,v)$ as defined in \cite{Orl09}. 
\end{example}

\begin{remark}
It may be of interest to  compare this equivalence to the one found in \cite{Orl09}.
\end{remark}

\subsection{Integral transforms}
In this section we use our results to recover and generalize a theorem of Baranovsky and Pecharich \cite{BP}.

\begin{definition} Let $X$ and $Y$ be smooth Deligne-Mumford stacks with $P \in \dqcoh{X \times Y}$. %such that the union of the support of the cohomology sheaves of $P$, $\op{supp}(P)$, is proper over $Y$
Denote the two projections by,
\begin{eqnarray*}
\pi_X: X \times Y \to X &\text{and} & \pi_Y :X \times Y \to Y.
\end{eqnarray*}

\noindent The induced \newterm{integral transform} is the functor,
%\begin{eqnarray*}
%\Phi_{P}:  \dqcoh{X} \to \dqcoh{Y}& , &A \mapsto \mathbf{R}{\pi_Y}_*(\pi_X^*A \otimes  P).
%\end{eqnarray*}
\[
\Phi_{P} :=  \mathbf{R}{\pi_Y}_* \circ (- \overset{\mathbf{L}}{\otimes}   P) \circ \mathbf{L}\pi_X^* : \dqcoh{X} \to \dqcoh{Y}
\]
The object $P$ is called the \newterm{kernel} of the transform $\Phi_{P}$.
\end{definition}

Given two kernels $P \in \dqcoh{X \times Y}, Q \in \dqcoh{Y \times Z}$ we define the convolution to be
\[
P \star Q := \mathbf{R}{\pi_{XZ}}_*(\pi_{XY}^* P  \otimes \pi_{YZ}^*  Q).
\]
It is a standard fact that $\Phi_Q \circ \Phi_P$ is naturally isomorphic to $\Phi_{P \star Q}$.  If one likes, they can take it as a particular case of Proposition~\ref{prop: FM transforms} below.

Let $w: X \to \mathbb A^1$ and $v: Y \to \mathbb A^1$ be morphisms.

\begin{definition}
A morphism $f: X \to Y$ is called \newterm{factored} with respect to $w,v$ if $w = v \circ f$.
\end{definition}

\begin{lemma}
If $f: X \to Y$ is factored then $f_*: \op{Qcoh }X \to \op{Qcoh }Y$ and $f^*: \op{Qcoh }Y \to \op{Qcoh }X$ are factored.
\end{lemma}

\begin{proof}
This follows immediately from the definitions.
\end{proof}

\begin{example}
Let $w : X \to \mathbb A^1$ and $v: Y \to \mathbb A^1$.  Then the natural projections of the fiber product $X \times_{\mathbb A^1} Y$ to $X$ and $Y$ are factored.
\end{example}

\begin{definition}
Let $w : X \to \mathbb A^1$ and $v: Y \to \mathbb A^1$.  Let $\pi_X, \pi_Y$ be the projections of $X \times_{\mathbb A^1} Y$ onto $X, Y$ respectively and let $P \in \dqcoh{X \times_{\mathbb A^1} Y}$.  The \newterm{factored integral transform} is the functor
%\begin{eqnarray*}
%  \Phi^f_{P}: \dabsFact{w} \to \dabsFact{v}& , &E \mapsto \mathbf{R}{\pi_Y}_*^f({\pi_X}^*)^f E \otimes  P)^f).
%\end{eqnarray*}
\[
\Phi_{P}^f :=  \mathbf{R}{\pi_Y}_*^f \circ (- \overset{\mathbf{L}}{\otimes}  P)^f \circ \mathbf{L}(\pi_X^*)^f : \dabsFact{w} \to \dabsfact{v}.
\]
%where  ${\pi_X}^*E \otimes  P := \operatorname{tot}({\pi_X}^* E^{-1} \otimes P, {\pi_X}^* E^{0} \otimes P, {\pi_X}^* \phi_E^{-1} \otimes \operatorname{Id}, {\pi_X}^* \phi_E^{0} \otimes \operatorname{Id})$.
%We call $P$ a \newterm{factored kernel}.
\end{definition}

\begin{proposition}[Base extension for factorizations] \label{prop: base change}
Consider a Cartesian square of factored morphisms
 \begin{center}
 \begin{tikzpicture}[description/.style={fill=white,inner sep=2pt}]
  \matrix (m) [matrix of math nodes, row sep=2em, column sep=2em, text height=1.5ex, text depth=0.25ex]
  { Z  &  Y  \\ 
    X  &  W \\ };
  \path[->,font=\scriptsize]
  (m-1-1) edge node[above] {$u'$} (m-1-2)
  (m-1-1) edge node[left] {$v'$} (m-2-1)
  (m-1-2) edge node[right] {$v$} (m-2-2)
  (m-2-1) edge node[above] {$u$} (m-2-2)
  ;
 \end{tikzpicture}
 \end{center}
Assume $u$ is flat.  Then there is a natural isomorphism between the composition of derived  factored functors
\[
 (u^*)^f \circ \mathbf{R} v_*^f  \cong   \mathbf{R} (v'_*)^f \circ (u'^*)^f
\]
\end{proposition}

\begin{proof}
Recall that the usual statement of flat base change states that we have a natural morphism
\begin{displaymath}
 u^* \circ v_* \to  v'_* \circ u'^*.
\end{displaymath}
which induces a natural isomorphism
\begin{displaymath}
 u^* \circ \mathbf{R} v_*  \cong   \mathbf{R} v'_* \circ u'^*.
\end{displaymath}
Consequently, we see that the image of injectives under $u'^*$ is $v'_*$-adapted. So, 
\begin{displaymath}
 \mathbf{R} (v'_* \circ u'^*) = \mathbf{R} v'_* \circ u'^*.
\end{displaymath}
Tautologically, 
\begin{displaymath}
 \mathbf{R} (u^* \circ v_*) = u^* \circ \mathbf{R} v_*.
\end{displaymath}
Thus, we are in exactly the situation of Lemma~\ref{lem: equal functors}.
\end{proof}

We will next need to prove a version of the projection formula.

\begin{proposition}[Projection formula for factorizations] \label{prop: projection formula}
Let $g : X \to Y $ be a factored morphism and $P$ be of complex of locally-free sheaves on $X$. There is a natural isomorphism between the composition of derived factored functors
\[
\mathbf{R} g_*^f \circ ( P  \overset{\mathbf{L}}{\otimes}_{\mathcal O_X} - )^f \circ \mathbf{L}(g^*)^f \cong (\mathbf{R}g_*P  \overset{\mathbf{L}}{\otimes}_{\mathcal O_Y} -)^f.
\]
\end{proposition}

\begin{proof}
 First, recall that the usual projection formula for $g$ gives a factored natural transformation 
 \begin{displaymath}
  \nu_{\mathcal E,\mathcal F}: g_*(\mathcal E \otimes g^* \mathcal F) \to g_* \mathcal E \otimes \mathcal F
 \end{displaymath}
 which is an isomorphism whenever $\mathcal E$ is quasi-coherent and $\mathcal F$ is locally-free. Using the induced natural transformation on factorizations with $\mathcal E$ a factorization with injective components and $\mathcal F$ a locally-free factorization yields the desired projection formula.
\end{proof}

\begin{proposition}\label{prop: FM transforms}
Let $P \in \dqcoh{X \times_{\mathbb A^1} Y}, Q \in \dqcoh{Y \times_{\mathbb A^1} Z}$ be complexes of vector bundles.  One has a natural isomorphism
\[
\Phi^f_{Q} \circ \Phi^f_{P} \cong \Phi^f_{P \star Q}.
\]
\end{proposition}

\begin{proof}
For notational simplicity, we let $E$ be a factorization with locally-free components.  This is obtained by replacing whatever factorization which such a $E$ via $\op{Q}_{\mathsf{C}} \circ \op{Q}_{\mathsf{C}}^{-1}$ where $\mathsf{C}$ is the class of locally-free sheaves.  We have natural isomorphisms
\begin{align*}
\Phi^f_{Q} \circ \Phi^f_{P}(E) & =   \mathbf{R}{\pi_Z}_*^f(Q \otimes (\pi_Y^*)^f(\mathbf{R}{\pi_Y}_*^f(P \otimes \pi_X^* E)^f ))^f \\
& \cong   \mathbf{R}{\pi_Z}_*^f(Q \otimes {\mathbf{R}\pi_{YZ}}_*^f({\pi_{XY}}^*)^f(P \otimes \pi_X^* E)^f )^f \\
& \cong   \mathbf{R}{\pi_Z}_*^f ({\mathbf{R}\pi_{YZ}}_*^f ({\pi_{YZ}^* Q \otimes (\pi_{XY}}^*)^f 	(P \otimes \pi_X^* E)^f )^f \\
& \cong   \mathbf{R}{\pi_Z}_*^f ({\mathbf{R}\pi_{XZ}}_*^f ({\pi_{YZ}^* Q \otimes (\pi_{XY}}^*)^f (P \otimes \pi_X^* E)^f )^f  \\
& \cong   \mathbf{R}{\pi_Z}_*^f ({\mathbf{R}\pi_{XZ}}_*^f ({\pi_{YZ}^* Q \otimes \pi_{XY}}^* P \otimes (\pi_{XY}^*)^f (\pi_X^*)^f E))^f \\
& \cong   \mathbf{R}{\pi_Z}_*^f ({\mathbf{R}\pi_{XZ}}_*^f ({\pi_{YZ}^* Q \otimes \pi_{XY}}^* P \otimes (\pi_{XZ}^*)^f (\pi_X^*)^f E))^f \\
& \cong   \mathbf{R}{\pi_Z}_*^f ({\mathbf{R}\pi_{XZ}}_*^f ({\pi_{YZ}^* Q \otimes \pi_{XY}}^* P )\otimes (\pi_X^*)^f E))^f \\
& \cong   \mathbf{R}{\pi_Z}_*^f (P \star Q \otimes (\pi_X^*)^f E))^f \\
& \cong   \Phi^f_{P \star Q}(E). \\
\end{align*}

The first line is by definition.  The second line is Proposition~\ref{prop: base change}.  The third line uses Proposition~\ref{prop: projection formula}.  The fourth line uses the fact that 
\begin{align*}
\mathbf{R}{\pi_Z}_*^f  \circ \mathbf{R}{\pi_{YZ}}_*^f & \cong  \mathbf{R}(({\pi_Z} \circ {\pi_{YZ}})_*^f ) \\
& \cong \mathbf{R}(({\pi_Z} \circ {\pi_{XZ}})_*^f ) \\
& \cong \mathbf{R}{\pi_Z}_*^f  \circ \mathbf{R}{\pi_{XZ}}_*^f 
\end{align*}
where the first and third lines are Corollary~\ref{cor: composition of derived functors} and the second line is just a natural isomorphism of functors at the Abelian level.
The fifth line uses the isomorphism of functors at the Abelian level
\[
\pi_{XY}^*( - \otimes - ) \cong \pi_{XY}^*(-) \otimes \pi_{XY}^*(-).
\]
The sixth line uses the isomorphism of functors at the Abelian level
\[
\pi_{XZ}^* \circ \pi_X^* \cong \pi_{XY}^*\circ\pi_X^*.
\]  The seventh line is Proposition~\ref{prop: projection formula} again.  The rest is by definition.  
\end{proof}
  
\begin{lemma} \label{lemma: we have the diagonal factorization}
 There is a natural isomorphism
 \begin{displaymath}
  \Phi_{\O_{\Delta}}^f \cong \op{Id}.
 \end{displaymath}
\end{lemma}

\begin{proof}
As in the above proof, we let $E$ be a factorization with locally-free components.  This is obtained by replacing whatever factorization which such a $E$ via $\op{Q}_{\mathsf{C}} \circ \op{Q}_{\mathsf{C}}^{-1}$ where $\mathsf{C}$ is the class of locally-free sheaves. We have natural isomorphisms
\begin{align*}
  \mathbf{R}\pi_{2*}^f ( \O_\Delta {\otimes}_{\mathcal O_{X \times X}} (\pi_1^*)^f  E  )^f
   & \cong   \mathbf{R}\pi_{2*}^f ( \Delta_*^f   (\Delta^*)^f (\pi_1^*)^f  E  )  \\
   & \cong   \mathbf{R}(\pi_{2*} \circ \Delta_*)^f  (\Delta^*)^f (\pi_1^*)^f  E  ) \\
   & \cong   E.
\end{align*}
The first line is Proposition~\ref{prop: projection formula}.  The second line is from the isomorphism of functors
\[
 (\O_\Delta \otimes_{\O_{X \times X}} - ) \cong \Delta_* \circ \Delta^*.
\]
Let third line is the isomorphism of functors
\[
 {\pi_2}_* \circ \Delta_* \cong \op{Id}
\]
and
\[
 {\Delta}^* \circ \pi_1^* \cong \op{Id}.
\]
\end{proof}

\begin{theorem}
Let $P \in \dqcoh{X \times_{\mathbb A^1} Y}$ and assume $\Phi_P$ is fully-faithful.   Then $\Phi^f_{P}$ is fully-faithful.
\end{theorem}  

\begin{proof}
Let $i : X \times_{\mathbb A^1} Y \to X \times Y$ be the natural map.
The functor $\Phi_{P}$ is fully-faithful if and only if $i_*Q \star i_*P \cong \O_{\Delta_X}$
 where
\[
Q := \mathbf{R}\mathcal{H}om(P, \O_{X \times _{\mathbb A^1} Y}) \otimes \pi_X^*\omega_X.
\]
so that
\[
i_* Q \cong \mathbf{R}\mathcal{H}om(i_*P, \O_{X \times Y}) \otimes p^*\omega_X.
\]
by Grothendieck duality.

Therefore, there are natural isomorphisms,
\begin{align*}
\Phi^f_{Q} \circ \Phi^f_{P} & \cong \Phi^f_{\O_{\Delta_X}} \\
& \cong \op{Id}_{\dabsFact{X}}
\end{align*}
where the first line comes from  Proposition~\ref{prop: FM transforms} and the second line comes from Lemma~\ref{lemma: we have the diagonal factorization}. 
\end{proof}

\begin{theorem}[Baranovsky, Pecharich]
Let $P \in \dqcoh{X \times_{\mathbb A^1} Y}$ and assume $\Phi_P$ is an equivalence.   Then $\Phi^f_{P}$ is an equivalence.
\end{theorem}
\begin{proof}
In the previous proof we showed that if $\Phi_{i_*Q \star i_*P} \cong \op{Id}_{\dqcoh{X}}$ then
$\Phi^f_{i_*Q \star i_*P} \cong \op{Id}_{\dabsFact{X}}$.  
The reverse direction is completely symmetric i.e.  if $\Phi_{i_*P \star i_*Q} \cong \op{Id}_{\dqcoh{Y}}$ then
$\Phi^f_{i_*P \star i_*Q} \cong \op{Id}_{\dabsFact{Y}}$.  This gives the result.
\end{proof}

%%%%%%%%%%%%%%%%%%%%%%%%%%%%%%%%%%%%%%%%%%%%%%%%%%%%%%%%%%%%%%%%%%%%%%%%%%%%

\end{document}